\documentclass{amsart}
 \usepackage{amsmath}
\usepackage{amssymb}
\usepackage{amsfonts}
 \usepackage{eucal}
      \makeatletter
     \def\section{\@startsection{section}{1}%
     \z@{.7\linespacing\@plus\linespacing}{.5\linespacing}%
     {\bfseries
     \centering
     }}
     \def\@secnumfont{\bfseries}
     \makeatother
\usepackage{a4wide}
\newtheorem{theorem}{Theorem}[section]
\newtheorem{lemma}[theorem]{Lemma}
\newtheorem{corollary}[theorem]{Corollary}
\newtheorem{proposition}[theorem]{Proposition}

\theoremstyle{definition}

\theoremstyle{remark}
\newtheorem{remark}[theorem]{Remark}

\numberwithin{equation}{section}

\def \a{{\alpha}}
\def \b{{\beta}}
\def \D{{\Delta}}

\def \d{{\delta}}
\def \e{{\varepsilon}}

\def \g{{\gamma}}

\def \l{{\lambda}}

\def \o{{\omega}}
\def \O{{\Omega}}
\def \p{{\varphi}}
\def \t{{\vartheta}}

\def \m{{\mu}}
\def \s{{\sigma}}

\def \N{{\bf N}}

\def \P{{\bf P}}

\def \qq{{\qquad}}
\def \R{{\bf R}}

\def \T{{\bf T}}

\def \uc{{\underline{c}}}

\def \Z{{\bf Z}}

\def \dd{{\rm d}}

\def \noi{{\noindent}}

\def \T{{\mathbb T}}

\def\P{{\mathbb P}}

\def\R{{\mathbb R}}

\def\Z{{\mathbb Z}}

\def\N{{\mathbb N}}

 \font\gum= cmbx10 at 10 pt

at  8 pt
\scrollmode

\font\gsec= cmb10 at 10 pt
   at 10,4  pt at  8,2 pt
  \scrollmode
\font\gsec= cmb10 at 9,8  pt
 at 8,2  pt
 
\begin{document}

  \title[Convergence of series of dilated functions]{\gum  An arithmetical approach  to the convergence problem  of series of dilated
 functions and its connection with the Riemann Zeta function} 
  \author{ Michel J.\,G.  WEBER}
 \address{Michel Weber: IRMA, 10  rue du G\'en\'eral Zimmer, 67084
 Strasbourg Cedex, France}
  \email{michel.weber@math.unistra.fr \!}


 \begin{abstract}
Given a periodic function $f$, we study the convergence almost everywhere and in 
norm   of the series  $\sum_{k} c_k f(kx)$. Let  $f(x)= 
\sum_{m=1}^\infty a_m \sin {2\pi m x}$ where $\sum_{m=1}^\infty  a_{m }^2d(m) <\infty$ and $d(m)=\sum_{d|m} 1$, and let
$f_n(x) = f(nx)$. We show  by using a new
decomposition of squared sums that for any $K\subset \N$ finite,    
$ \|\sum_{k\in K} c_k f_k \|_2^2    \le    ( \sum_{m=1}^\infty a_{m }^2 d(m)
 )    \sum_{k\in K }   c_{k}^2d(k^2)$. If $f^s (x)= \sum_{j=1}^\infty \frac{\sin 2\pi jx}{j^s}$, 
 $s>1/2$,   by only using
  elementary  Dirichlet   convolution calculus, we show that for $0< \e\le 2s-1$,
$\zeta(2s)^{-1} \|\sum_{k\in K}  c_k f^s_k \|_2^2   \le  \frac{1+\e}{\e }  (\sum_{k \in K}
  |c_k|^2  \s_{ 1+\e-2s}(k)  )$, where $\s_h(n)=\sum_{d|n}d^h$. 
From this we deduce that  if $f\in {\rm BV}(\T)$,
$\langle f,1\rangle=0$ and  
    
$$\sum_{k}  c_k^2\frac{(\log\log k)^4}{(\log\log \log k)^2}  <\infty,$$ then the series $ \sum_{k } c_kf_k$ converges almost
everywhere. 
 This   slightly improves a recent result, depending on a fine analysis on the polydisc
(\cite{ABS}, th.3) ($n_k=k$), where it was assumed that  
$    \sum_{k} c_k^2 \, (\log\log k)^\g     
$
converges for some $\g>4$. We further show that the same conclusion holds under the arithmetical condition
$$\sum_{  k }
   c_k^2  (\log\log k)^{2 + b} \s_{ -1+\frac{1}{(\log\log k)^{  b/3}}  }(k)      <\infty,$$ for some $b>0$,    or if $ \sum_{ k} c_k^2  d(k^2) (\log\log
k)^2   <\infty$. We  also derive from a recent result of Hilberdink an $\O$-result for the Riemann Zeta function involving factor closed sets. As an
application we find that simple conditions on $T$ and $\nu$ ensuring that for any   $\s>1/2$, $0\le \e <\s$, we have 
  \begin{eqnarray*}
    \max_{1\le t\le T}  |\zeta(\s + it)| 
   &  \ge & C(\s) \big(
\frac{1}{ \s_{-2\e} (\nu)}\sum_{n|\nu}\frac{\s_{-s+\e}( n)^2}{n^{2\e}} \big)^{1/2}
    ,
\end{eqnarray*}
 We finally prove an important  complementary result to Wintner's  famous characterization of mean convergence of series
$\sum_{k=0}^\infty c_k f_k $.    
\end{abstract}

\maketitle


\section{\bf Introduction}\label{s1}
Given a periodic function $f$ and an increasing sequence $\mathcal N= \{n_k,k\ge 1\} $ of positive integers, one can formally
define  the series 
$\sum_{k=1}^\infty c_k f(n_kx)$ and ask    under which conditions     this series  converges     in
norm or almost everywhere, for instance for any real coefficient sequence $\uc=\{c_k,k\ge 1\}\in \ell^2(\N)$. This is one of the oldest
and   most central problems  in the theory of systems of dilated sums.  We only  briefly outline the kind of results
obtained.   First studies were made  at the beginning of the last century
 (see Jerosch and Weyl \cite{JW} where a.e. convergence is obtained under growth conditions on coefficients and Fourier coefficients of $f$), parallel to similar ones for the
trigonometrical system. This partly  explains why  until Carleson's famous proof of Lusin's hypothesis,   the results obtained 
essentially concerned  functions with slowly growing modulus or integral modulus of continuity and/or sequences
$\mathcal N$ verifying the classical Hadamard gap condition: ${n_{k+1}}/{n_k}\ge q>1$ for all $k$. Carleson's result triggered a new
interest, permitting substancial  progresses in this direction,  under the main impulse of Russian analysts, among them Gaposhkin and
later by Berkes. We refer to
\cite{BW1} for more details and references. Then
the attention to these  problems  
  declined until very recently where there is a renewed activity, notably concerning their connection with some questions ($\O$-results) on the Riemann Zeta function.  \vskip 2 pt  In analogy with parallel questions concerning
   partial sums
$\sum_{k=1}^n   f(kx)$,
$n=1,2,\ldots$, strong law of large numbers,   studied by G\'al, Koksma (see also \cite{BW2}), 
and  law of the iterated logarithm,  central limit theorem, invariance principle, much explored by Erd\"os,    Berkes and
Philipp,  and Gaposhkin 
   notably,
recent works show that the arithmetical nature of the support of the coefficient sequence, as well as the analytic nature of $f$, interact in a
complex way in the study of the convergence   almost everywhere and in norm   of these series.  The part of the theory devoted to
individual results, namely the search of convergence conditions linking $f$, $\mathcal N$ and $\uc$ is, to say the least,
barely investigated. Our main concern in this work is precisely the search of individual conditions ensuring the almost everywhere
convergence of the series 
$\sum_{k=1}^\infty c_k f(n_kx)$. We propose  new approaches  for treating these questions.   
Notice before continuing, that the problem
under consideration  is a natural continuation of the study of the trigonometrical system, since by Carleson's result, the series
$\sum_{k} c_k f(n_kx)$ converges almost everywhere for any trigonometrical polynomial $f$. And   this is in fact a convergence problem
that can be put inside the study of the   two-indices    trigonometrical system with    $\{e_{jk}, j,k\ge 1\}$ where we denote
$e(x)= e^{2i\pi x}$,
$e_n(x)= e(nx)$, $n\ge 1$. Let
$\T=\R\slash
\Z=[0,1[$. Let
$f(x)\sim
\sum_{j=1}^\infty  a_j e_j(x) $.  Let  
$f_n(x) = f(nx)$,
$n\in
\N$.  We assume throughout that 
$$f\in L^2(\T), \qq \qq \langle f,1\rangle =0 .$$ 

\vskip 2 pt A key  
preliminary step naturally consists with the search of bounds of
$
\|\sum_{k\in K}  c_k f_k
\|_2$  integrating in their formulation the arithmetical  structure of
$K$. That question   has      received a satisfactory answer only for specific cases. 
   In this work  we propose an 
approach   based on  
  elementary  Dirichlet   convolution calculus  and on a new decomposition of squared sums.  Although quite natural in regard of the posed problem, it
seems  at least to our knewledge, that  this  direction was  not prospected before, apart in the recent works \cite{We}, \cite{BW}. 
\vskip 2 pt
We show that our approach is strong enough to recover and even slightly improve a recent a.e. convergence result \cite{ABS} (Theorem 3) 
in the case $\mathcal N=\N$ without using analysis on the polydisc, see Theorem \ref{t1a}. We obtain  norm estimates 
of arithmetical type of  $\sum_{k} c_k f(kx)$  and also a related $\O$-result for the Riemann Zeta function involving factor closed sets. We begin with
stating and commenting them. Next we will state almost everywhere convergence results derived  from these estimates. The remainding part of the paper is
devoted to  the proofs of the results.

 \vskip 3pt \noi {\gsec Notation.} We write $\log\log x= \log \log (x\vee e^e)$, $\log\log\log x= \log\log \log (x\vee e^{e^e})$, $x>0$. 
 \vskip 3pt 
\section{Arithmetical Results} 
\subsection{A general arithmetical norm estimate.}  Let  
$d(n)$ be the  divisor function, namely the number of divisors of   $n$. 
  \begin{theorem}\label{c1} Assume that $ \sum_{m=1}^\infty a_{m }^2 d(m)<\infty$. Then, for any finite set $K$ of positive integers,
\begin{eqnarray*}  \big\|\sum_{k\in K} c_k f_k\big\|_2^2   &\le & \Big( \sum_{m=1}^\infty a_{m }^2 d(m)
\Big)    
\sum_{k
\in K }   c_{k}^2d(k^2).
    \end{eqnarray*}
 \end{theorem}

The presence of the arithmetical factor $d(k^2)$ comes from formula (\ref{theta}). In   \cite{We}, we recently showed a similar estimate, however
restricted to sets
$K$ such that   
$K\subset ]e^r , e^{r+1}]$ for some
 integer $r$. Then,  
\begin{eqnarray}\label{s01}  \big\|\sum_{k\in K}c_k f_k\big\|_2^2   \le   \Big(  \sum_{\nu=1 }^\infty a_{ \nu }^2\D( \nu
 ) \Big)\sum_{ k  \in K  }c_k^2 d(k)  ,
  \end{eqnarray}  
where $\D(v)$ is  Hooley's  Delta function, 
$$\D(v)=\sup_{u\in \R}\sum_{d|v \atop
u<d\le eu} 1.$$
This one can be used to prove that under the conditions
$$A=\sum_{\nu\ge 1 }  a_{ \nu }^2\D(  \nu  ) <\infty, \qq \quad  
 B= \sum_{j  } c_j^2d(j)(\log j)^2 <\infty$$
     the series $\sum_{k=0}^\infty c_k f_k(  x)$ converges for almost all $x $. A slightly weaker result was established in \cite{We}
(see Theorem 1.1). Condition $A<\infty$ is very weak. As by \cite{T1},
  $$ \frac1{x}  \sum_{n\le x} \D(n)= \mathcal O\big( e^{c \sqrt{\log\log x \cdot
\log\log\log  x}}
\big) $$ for a suitable constant $c>0$ (whereas $\frac1{x}  \sum_{v\le x} d(v) \sim \log x$), 
it reduces when the Fourier coefficients are monotonic to 
\begin{equation} \label{hooley1} \sum_{\nu\ge 1 }  a_{ \nu }^2e^{c \sqrt{\log\log \nu \cdot
\log\log\log  \nu} }<\infty.
\end{equation} 
 Theorem
\ref{c1} is  deduced from a more general result. Introduce the necessary notation.  Let 
$$A_k=\sum_{\nu=1}^\infty a_{ {  {\nu  k} } }^2 
.$$
  Let    $\zeta_h$ denotes the arithmetic function defined by $\zeta_h (n) = n^h$ for all positive $n$.
In particular
$\zeta_0(n)=1$ for all
$n$. Let   $\theta(n) $ denotes the number of  squarefree divisors
of $n$.
Then $\theta(n)=2^{\o(n)}$  where $\o(n)$ is the prime divisor function, and by Mertens estimate, $\sum_{ k\le x }   2^{\o(k)} =C x
\log x   +{\mathcal O}(x)$, $x\ge 2$, where $C$ is some positive constant (\cite{CMS}, p.70). 
\vskip 1pt 
Given   $K\subset \N$, we note $F(K)=\{ d\ge 1;
\exists k\in K: d|k\}$. If $K$ is factor closed ($d|k\Rightarrow d\in K$  for all $k\in K$), then $F(K)=K$. 
These sets are usually termed FC sets (see \cite{CS} \S 3.3, \cite{HSW}). 
Typical examples are 
$  \{1,\ldots, n\}$  or   sets of mutually coprime integers. 
  Recall that if $\psi, \phi$ are arithmetical functions,   the Dirichlet convolution $\psi*\phi$ is defined by $ \psi*\phi (n)=
\sum_{d|n} \psi(d)
\phi(n/d)$.
  \begin{theorem}
\label{t1}
   Let $\psi$ be any
  arithmetical function taking only positive values. 
\vskip 2 pt
\noi i) For any finite set $K$ of positive integers, 
 \begin{eqnarray*}\label{s1b} \big\|\sum_{k\in K} c_k f_k\big\|_2^2& \le&     \sum_{d\in F(K)}\Big(\sum_{k \in
K\atop d|k}   c_{k}^2\psi(\frac{k}{d})\Big)
   \Big(\sum_{k \in K\atop d|k} 
\frac{A_{  \frac{k}{d}  }}{
\psi(\frac{k}{d})} \theta(\frac{k}{d}) 
\Big)  .
  \end{eqnarray*}
 \noi ii) In particular,
\begin{eqnarray*}\big\|\sum_{k\in K} c_k f_k\big\|_2^2&\le & B\, \sum_{k\in K} c_k^2\psi*\zeta_0(k ),
\end{eqnarray*} 
where  \begin{eqnarray*} B  =\sup_{ d\in F(K)} \Big(\sum_{k \in K\atop d|k} 
\frac{A_{  \frac{k}{d}  }}{
\psi(\frac{k}{d})} \theta(\frac{k}{d})
\Big)   
 <\infty.
  \end{eqnarray*} 
    \end{theorem}
  
  Choose for instance $\psi= \theta$ and note (\cite{MC} formula 1.54)) that
\begin{eqnarray}\label{theta}\psi*\zeta_0(k )=\sum_{d| k} \theta(d) = d(k^2).
  \end{eqnarray}
Let $d\in F(K)$, then
\begin{eqnarray*}\sum_{k \in K\atop d|k} 
\frac{A_{  \frac{k}{d}  }}{
\psi(\frac{k}{d})} \theta(\frac{k}{d})=\sum_{k \in K\atop d|k} 
\sum_{\nu=1}^\infty a_{ {  {\nu   \frac{k}{d} } } }^2=\sum_{m=1}^\infty a_{m }^2\Big(\!\sum_{k \in K\atop d|k,  \frac{k}{d}|m} 
  1
\Big)=\sum_{m=1}^\infty a_{m }^2\Big(\sum_{   jd\in K\atop j|m} 1
\Big) \le \sum_{m=1}^\infty a_{m }^2 d(m).
\end{eqnarray*}  
Since it is true for any $d\in F(K)$, we deduce that $B\le \sum_{m=1}^\infty a_{m }^2 d(m)$, and so   Theorem \ref{c1} follows  from
Theorem
\ref{t1}.
     
\subsection{Related classes of  arithmetical quadratic forms.} 
  The question  of the reduction of a quadratic form whose coefficients are a function of the greater commun divisor of
their indices 
$$ X= \sum_{i,j}  x_ix_j  F((i,j))$$
       was  considered    long ago by Ces\`aro \cite{Ce1}, \cite{Ce2} in 1885-1886, after the works of  Smith \cite{S}
and Mansion  in 1875-1876 who calculated     their determinant (Ces\`aro also calculated other classes of arithmetical determinants). 
These quadratic forms are in turn directly related to an important class of functions (see \ref{funct})) whose associated systems of dilated functions was
recently intensively studied.    Other authors, among them Jacobstahl
\cite{J},  Carlitz 
\cite{C}, investigated before this problem (see
the survey on GCD matrices by Haukkanen, Wang and Sillanp\"a\"a \cite{HSW}  for more references). In the present case, the reduction takes the following
form, 
\begin{equation}  \label{reds} \sum_{k,\ell =1}^n c_kc_\ell \frac{(k,\ell)^{2s}}{j^s \ell^s}=\sum_ {i=1}^n  J_{2s}(i)\Big( \sum_{k=1}^n c_k k^{-s} {\bf
1}_{i|k} \Big)^2.
\end{equation}
Here   $J_{2s}(i)=i^{2s} \prod_{p|i} (1-p^{-2s})$ is the generalized Euler totient function (see Section \ref{s3}). This formula, which is used in \cite{BW}
(see Lemma 1.1), was already known by Ces\`aro. Obviously, (\ref{reds}) remains   true when replacing 
$\{1,\ldots,n\}$ by a factor-closed set. A dual problem is   M\"obius inversion of a family of vectors (with Gram matrix $\{F((i,j))\}_{i,j} $). Recent
related works are in  Balazard and Saias
\cite{BS}, Br\'emont
\cite{Br}.

\vskip 3pt
In matrix form, this can be condensed in the following   proposition, which  generalizes Proposition 2.2 in \cite{Br} based on Carlitz
Lemma, see also Li's representation of GCD matrices \cite{Li} and
\cite{HL}. As the proof is elementary and short,  we shall give it right after some necessary complementary remarks. 
\begin{proposition}\label{mov} Let     $T=\big(t_{i,j}\big)_{n\times n}$ and   $\check T=\big(\check t_{i,j}\big)_{n\times n}$ be matrices defined
by 
\begin{equation} t_{i,j}= \d_i\theta_j {\bf 1}_{ i|j },  \qq \check t_{i,j}= \frac{1}{\theta_i\d_j} {\bf 1}_{ i|j
}\m\big(\frac{j}{i}\big) ,
\qq i=1,\ldots, n,\end{equation}  where     $\d_i , \,   \theta_i  ,  i=1,\ldots, n$ are     real numbers   verifying   $  \d_i>0,
\theta_i\not= 0$ for all $ i $.
 Then, \vskip 3pt
 {\rm (a)} 
$T$ is invertible and 
$T^{-1}=    
\check T$. 
\vskip 2pt
  {\rm (b)} Let $H_m, m=1,\ldots, n$ be real numbers defined as follows
\begin{equation}H_m=\sum_{k|m} \d_k^2    ,\quad
m=1,\ldots, n.
\end{equation}
Then ${}^tT T=A  $ where $A= \big(\theta_i\theta_j H_{(i,j)}\big)_{n\times n}$. Further $A$ is positive definite.
\vskip 2pt  {\rm (c)} Let $ G=(g_i)_{1\le i\le n} $ be   vectors    in  an inner product space     such that ${\rm Gram}(G)= A$. Then  
$F={}^t\check T G =(f_i)_{1\le i\le n}$ has orthonormal components.
\vskip 2pt  {\rm (d)} For any reals $c_i $, we have 
\begin{equation} \sum_{k=1}^n c_k g_k = \sum_{i=1}^n b_i f_i, \quad {\rm where} \quad b_i =\d_i\big(\sum_{k=1}^n c_k \theta_k {\bf
1}_{i|k} \big), \ i=1,\ldots, n.
\end{equation} In particular,
 $$\big\|\sum_{k=1}^n c_k g_k\big\|^2= \sum_{i=1}^n    b_i^2. $$

\end{proposition} 
\begin{remark} We recall   that  positive semi-definite matrices are always Gram matrices (of vectors in  an inner
product space), hence the existence of
$G$ in   (c). Further, a matrix $B$ is positive  definite if and only if there exists a non-singular lower triangular matrix $L$ such that
$A= L{}^t L$, see  \cite{HJ}, Corollary 7.2.9. Furthermore,  by the M\"obius inversion formula,
$$H_m= \sum_{k|m} \d_k^2 ,\quad
m=1,\ldots, n\qq \Longleftrightarrow \qq \d_k^2= \sum_{  \ell |k} \mu\big(\frac{k}{\ell}\big)H_\ell,\quad
k=1,\ldots, n. $$
\end{remark}
\begin{remark}
  Take    $\theta_j= j^{-s}$, $H(x)= x^{2s}$. Then $\d_k= \sqrt{J_{2s}(k)}$ and 
$$a_{i,j}=\frac{(i,j)^{2s}}{i^s j^s} , \qq \quad b_{i}=\sqrt{J_{2s}(i)}\Big(\sum_{k=1}^n c_k k^{-s} {\bf 1}_{i|k} \Big)  .$$
Hence (\ref{reds}). 
Next  consider  the   class of functions   introduced in \cite{LS},  
\begin{equation}\label{funct}  f^s (x)= \sum_{j=1}^\infty \frac{\sin 2\pi jx}{j^s} 
\end{equation}
where  
$s>1/2$. 
 Recall that 
\begin{equation*}  \langle f^s_k, f^s_\ell\rangle=\zeta(2s) \frac{(k,\ell)^{2s}}{k^s\ell^s} .
\end{equation*}
And so \begin{equation}  \label{HS1}\big\|\sum_{k=1}^n c_k f^s_{k}\big\|_2^2=\zeta(2s)\sum_{k,\ell=1}^n
 c_k   c_\ell\frac{(k,\ell)^{2s}}{k^s\ell^s} .
 \end{equation}
  This class of functions was recently much studied. We  begin with examinating    

Further,  the system  
\begin{equation} h_i^{s}= \frac{1}{\sqrt{J_{2s}(i)}}\sum_{j|i}   j^s    \m\big( \frac{i}{j}\big)f_j^{s}\qq i=1,\ldots  
\end{equation}
is orthonormal and 
\begin{equation} \sum_{k=1}^n c_k f^{s}_k = \sum_{i=1}^n b_i h^{s}_i, 
\quad {\rm where} \quad b_i =\sqrt{J_{2s}(i)}\big(\sum_{k=1}^n c_k
k^{-s} {\bf 1}_{i|k} \big), \ i=1,\ldots, n, \quad n\ge 1.
\end{equation} 

This is Br\'emont's result,  who deduced the following characterization:
\vskip 2pt  {\it The series $\sum_{k } c_kf^{s}_k$ converges in
$L^2(\T)$ if and only if the following
uniformity condition holds,}
\begin{equation}\label{uc} \lim_{n\to \infty}\sup_{N>n} \sum_ {i=1}^\infty J_{2s}(i)\Big( \sum_{k  =N+1}^\infty c_k k^{-s} {\bf
1}_{i|k} \Big)^2 \ =\ 0.
 \end{equation}
Notice that by  assumption and Cauchy-Schwarz's inequality, the series $\sum_{k\ge 1}|c_k|k^{-s} $ is convergent. 
See \cite{Br}, Proposition 2.2 and Corollary 2.3.  Although satisfactory, (\ref{uc})  is however  implicit, and it would be desirable to find a more concrete
characterization, namely depending more directly on the coefficients $(c_k)_k$. 
As a (non trivial) application, it is showed in \cite{Br}, that  $L^2(\T)$-convergence holds if $|c_k|\le \d(k)$ where $\d$ is multiplicative and
$\sum_n\d^2(n)<\infty$. \end{remark} 
\begin{proof}[Proof of Proposition \ref{mov}] Let $I$ denote the $n\times n$ identity matrix. 
\item (a) We have,
$$\sum_{k=1}^n \check t_{i,k} t_{k,j}=\sum_{k=1}^n\frac{1}{\theta_i\delta_k}{\bf 1}_{i|k}\m\big(\frac{k}{i}\big) \d_k\theta_j {\bf
1}_{k|j}=\frac{\theta_j}{\theta_i}\sum_{k=1}^n {\bf 1}_{i|k}\m\big(\frac{k}{i}\big)   {\bf
1}_{k|j} = \frac{\theta_j}{\theta_i}{\bf 1}_{i|j}\sum_{  \ell |\frac{j}{i}} \m(\ell)=0,
$$
if $i\not =j$, since $\sum_{  \ell |m} \m(\ell)=0$, if $m\ge 2$. Hence $\check T T= I$, and similarly $  T \check T= I$.
\item (b) We compute  the $(i,j)$-th entry of  ${}^t T T$. 
$$ \sum_{k=1}^n t_{k,i}t_{k,j} =\theta_i \theta_j \sum_{k|(i,j)} \d_k^2 =\theta_i \theta_j H_{(i,j)}. $$
  by the M\"obius inversion formula.
We also have $\det A=( \det T)^2=\prod_{i=1}^n \d_i\theta_i\not= 0$. 
And if $X\not=0$, then ${}^t XAX={}^t YY> 0$ with $Y= TX\not=0$.
    \item (c)
\begin{eqnarray*}\langle f_i, f_j\rangle&= &\sum_{k=1}^n \sum_{\ell =1}^n \check t_{k,i}\check t_{\ell, j} \langle g_k, g_\ell\rangle 
= \sum_{k=1}^n \sum_{\ell =1}^n \check t_{k,i}\Big(  \sum_{u=1}^n t_{k,u}t_{u, \ell} \Big)\check t_{\ell, j}\cr 
&=&\sum_{k=1}^n\check t_{k,i} \sum_{u=1}^n t_{k,u}{\bf 1}_{u=j}  =\sum_{k=1}^n\check t_{k,i}
t_{k,j}={\bf 1}_{i=j}.\end{eqnarray*}
  \item (d)   As   $   G ={}^tTF$, we have  
 $t_{i,j}= \d_i\theta_j {\bf 1}_{ i|j }$
 $$\sum_{k=1}^n c_k g_k =\sum_{k=1}^n c_k \sum_{i=1}^n t_{i,k}f_i =\sum_{k=1}^n c_k\big(\theta_k \sum_{i=1}^n \d_i {\bf 1}_{i|k}
f_i\big)=
\sum_{i=1}^n  f_i
 \d_i\big(\sum_{k=1}^n c_k \theta_k {\bf 1}_{i|k} \big)=\sum_{i=1}^n  f_i b_i.$$
 \end{proof}
 \vskip 3 pt

Another approach was proposed by  Hilberdink \cite{H} who has estimated the sums
$$\sum_{k,\ell=1}^N
  \frac{(n_k ,n_\ell)^{2s}}{n_k^s n_\ell^s}  $$
when $n_k=k$ and obtained optimal bounds in this case. 
He showed that if $b_n=n^{-s} \sum_{d|n} d^{s}a_d$, then
\begin{equation}\label{hi}\sum_{n=1}^N  |b_n|^2= \sum_{m,n\le N} a_m   \overline{a}_n\frac{(m,n)^{2s}}{m^sn^s}{\sum_{k\le
N/[m,n]}}^{\hskip -13 pt*}\ \ \ \frac{1}{k^{2s}}.
\end{equation}
Here we introduce the symbol (${}^*$) to mean that the sum is $0$ when the summation index is  empty. And this requires  some restriction with respect to the original statement,  see
Proposition 3.1 and after in \cite{H}.  
More precisely, if  $a_n\ge 0$, the
right-term is less than
$\zeta(2s) \sum_{m,n\le N} a_m   \overline{a}_n\frac{(m,n)^{2s}}{m^sn^s}$. And a  similar lower bound occurs when   $a_n\ge 0$. 

\vskip 2 pt
When the $n_k$'s are arbitrary but distinct
positive integers,
 the initial result is due to G\'al \cite{Ga},
who showed for $s=1$ that
\begin{equation}\label{gal}\sum_{k,\ell=1}^N
  \frac{(n_k ,n_\ell)^{2 }}{n_k n_\ell} \le CN(\log\log N)^2, 
\end{equation}
where $C$ is an absolute constant, and moreover that estimate is optimal. It follows for this choice of values of
$n_k$ and by taking $c_k\equiv 1$ that in this case
\begin{equation}  \label{ext}\big\|\sum_{k=1}^N c_k f^1_{n_k}\big\|_2^2\ge C N(\log\log N)^2\gg  \sum_{k=1}^N c_k^2.
 \end{equation}
 
This is a famous result and   a few  explanatory words concerning the proof are necessary. G\'al's proof is based on the observation that the sum in (\ref{gal}) will be not maximal unless $\{n_1,\ldots, n_k\}$ is an FC set, 
namely $d|n_j\Rightarrow d=n_i$ for some $i$. Hence it follows that if the sum is maximal, then the corresponding $n_i$ are products of powers of at most $C \log N$ primes. 

This result was recently extended in \cite{ABS} to the case
$0<s<1$ (see also \cite{BoS} for   recent improvments, in the case $s=1/2$ notably)
 by representing these sums as Poisson integrals on the polydisc and by suitably modifying G\'al's combinatorial argument.   When
sieving the coefficients $c_k$   according to their order of magnitude, that estimate can be implemented and then becomes a  decisive
tool when $f$ has slowly decreasing Fourier coefficients, typically  when
$f=f^1$. That allowed the authors to
establish quite sharp results for the a.e. convergence of series $\sum_k c_kf^{1}_{n_k} $,  and in fact by a plain monotonicity argument
on the Fourier coeffcients, for any $f\in{\rm BV}(\T)$. The authors further extended their result to any $f\in{\rm Lip}_{1/2}(\T)$. 
These results are of relevance in the present work.
\vskip 3pt \noi {\gsec
Representation using Cauchy measures.} 
Notice before continuing that 
\begin{eqnarray*}
  \frac{(n_k ,n_\ell)^{2 s}}{n_k^s n_\ell^s}&=&\prod_{p} p^{-s|v_p(n_k)-v_p(n_\ell)|}
\end{eqnarray*}
where $v_p(n)$ is the $p$-valuation of $n$ (namely $p^{v_p(n)}||n$ and $v_p(n)=0$ if $p\not|n$). And from the relation    
  $ e^{-|\t|}= \int_\R  e^{i\t t}\frac{\dd t}{\pi( t^2+1)}$, 
follows that 
\begin{eqnarray*} \label{cor1}e^{-s|v_p(n_k)-v_p(n_\ell)|\log p}=\int_\R   \frac{1}{  p^{ iv_p(n_k)st}p^{-iv_p(n_\ell )st}}
\frac{\dd t}{\pi( t^2+1)}. 
\end{eqnarray*}  
So that  \begin{eqnarray} \label{Cauchy}\sum_{k,\ell =1}^N c_k\overline{c}_\ell \frac{(n_k ,n_\ell)^{2 s}}{n_k^s n_\ell^s}
&=&\sum_{k,\ell =1}^N c_k\overline{c}_\ell  \int_{\R^\infty} 
  \prod_{p}\Big(\frac{1}{  p^{ iv_p(n_k)st_p}p^{-iv_p(n_\ell )st_p}} \frac{\dd t_p}{\pi( t_p^2+1)}\Big)
  \cr 
  &=&  \int_{{\R}^\infty}  
 \Big| \sum_{k =1}^N   \frac{c_k}{  \prod_{p}p^{ iv_p(n_k)st_p}}\Big|^2\prod_{p} \frac{\dd t_p}{\pi( t_p^2+1)}, 
\end{eqnarray} 
namely, the sum directly expresses as a squared norm with respect  to the infinite Cauchy measure. 
\subsection{A new arithmetical quadratic estimate.} It turns up that even for this specific class of functions,  another much simpler
device can be used, based on Dirichlet convolution calculus, which also leads, at least when
$n_k=k$, to slightly sharper convergence results. The basic tool, which we are
going to state now,   provides a new estimate of $\|\sum_{k\in K} c_k f^s_{ k}\|$, $K$ arbitrary. This estimate is of individual type, in the sense that
it is expressed by means of the values taken on $K$ by some elementary arithmetical functions. 
 Let for  $u\in \R$, $\sigma_u (k)=\sum_{d|k} d^u$. In particular $ \sigma_0=d $, $ \sigma_1=\s $ is the usual sum-divisor
function and   $\s_{-\a}(n)=n^{-\a} \s_\a(n)  $.
 \begin{theorem} \label{p1}Let $s>0$ and $0\le \tau\le 2s$. Let also $\psi_1(u)>0$  be non-decreasing.  Then for any finite set $K$ of integers,
\begin{eqnarray*} \label{oa}
\sum_{k,\ell\in K} c_k   c_\ell\frac{(k,\ell)^{2s}}{k^s\ell^s}  &\le& \Big(\sum_{u \in F(K)} \frac{1}{\psi_1(u)\s_\tau(u) }
\Big)\Big(\sum_{\nu \in K}
  c_\nu^2  \psi_1(\nu)\s_{ \tau-2s}(\nu) \Big). 
 \end{eqnarray*}

In particular, 
\vskip 3pt 
\noi \ {\rm (i)}  
\begin{eqnarray*} 
\sum_{k,\ell\in K}c_k   c_\ell
 \frac{(k,\ell)^{2s}}{k^s\ell^s}   &\le& M(K) \Big(\sum_{k \in K}
  |c_k|^2  \s_{ \tau-2s}(k) \Big) , 
 \end{eqnarray*}
with
$$M(K)=  \sum_{k \in F(K)} \frac{1}{\s_\tau(k) }   .$$

\noi \ {\rm (ii)} \begin{eqnarray*}
\sum_{k,\ell\in K} c_k   c_\ell\frac{(k,\ell)^{2s}}{k^s\ell^s}  &\le &  \Big(\sum_{u \in F(K)} \frac{1}{\overline{\s}_{-s}(u)\s_s(u) }
\Big)\Big(\sum_{\nu \in K}
  c_\nu^2 \overline{\s}_{-s}(u)\s_{-s}(u) \Big), 
 \end{eqnarray*}
where $\overline{\s}_{-s}(u):=\max\{\s_{-s}(v),  1\le v\le u\}$.
\vskip 3pt 
\noi \ {\rm (iii)} Further
\begin{eqnarray*} 
\sum_{k,\ell\in K} c_k   c_\ell\frac{(k,\ell)^{2}}{k\ell} 
&\le&
 \frac{\pi^2}{6 } \Big(\sum_{u \in F(K)} \frac{\p(u)}{{u^2}\log\log u } \Big)\Big(\sum_{\nu \in K}
  c_\nu^2 \s_{-1}(\nu) \log\log \nu \Big),
 \end{eqnarray*}
  \end{theorem}  
  As an immediate consequence, 
  we get recalling  (\ref{HS1}),
  \begin{corollary} \label{r1} Let $s>1/2$.    Then for any finite  set $K$,\begin{eqnarray*} 
\zeta(2s)^{-1}\big\|\sum_{k\in K}  c_k f^s_k\big\|_2^2  &\le&  \inf_{0< \e\le 2s-1} \frac{1+\e}{\e } \, \sum_{k \in K}
  |c_k|^2  \s_{ 1+\e-2s}(k)   . 
 \end{eqnarray*}
  \end{corollary}
\begin{proof}  Let indeed $0< \e\le 2s-1$ and take $\tau = 1+\e$. From the obvious inequality $\s_\tau(k)\ge k^\tau$, follows that
$$M(K)\le   \sum_{k \in F(K)}\frac{1}{k^{1+\e} } \le  \sum_{k \ge 1} \frac{1}{k^{1+\e} }\le \frac{1+\e}{\e } .$$
So that Corollary \ref{r1} just follows from assertion (ii) of Theorem \ref{p1}. 
\end{proof} 
\begin{remark}Letting for instance $s=1$ and using
monotonicity of the Fourier coefficients, Corollary 
\ref{r1} implies in particular:  
\vskip 1 pt For any $f\in {\rm BV}(\T)$ with $\int_\T f\dd \l = 0$, 
 \begin{eqnarray*} 
 \big\|\sum_{k\in K}  c_k f_k\big\|_2^2  &\le& C(f) \inf_{0< \e\le 2s-1} \frac{1+\e}{\e }\, \sum_{k \in K}
  |c_k|^2  \s_{ -1+\e }(k)   . 
 \end{eqnarray*} 
And $C(f)$ depends on $f$ only.
\end{remark}
 
We derive from Corollary \ref{r1}  a  sharp arithmetical sufficient condition for the in-norm
convergence of the series
$\sum_{k\ge 1}  c_k f^s_k$.
 \begin{corollary} \label{r2} Let $s>1/2$. Assume that the following condition is fulfilled:
     \begin{eqnarray*} 
\hbox{For some $\e>0$,}\qq \sum_{k\ge 1}
  |c_k|^2  \s_{ 1+\e-2s}(k) <\infty . 
 \end{eqnarray*}
Then the series $\sum_{k\ge 1}  c_k f^s_k$ converges in $L^2(\T)$.   \end{corollary} 
 
\begin{proof}This is now straightforward, since by Corollary \ref{r1}
\begin{eqnarray*} 
 \sup_{n,m\ge N}\big\|\sum_{n\le k\le m}  c_k f^s_k\big\|_2^2  &\le& C(s,\e) \sum_{k\ge N}
  |c_k|^2  \s_{ 1+\e-2s}(k)  \to 0, 
 \end{eqnarray*}
as $N$ tends to infinity. Hence $\{\sum_{1\le k\le m}  c_k f^s_k, m\ge 1\}$ is a Cauchy sequence in $L^2(\T)$. 
\end{proof}See also the recent work \cite{ABSW} for a different proof. 

\begin{remark} By monotonicity  of Fourier coefficients, Corollary \ref{r2} immediately extends with no change to functions
$f\sim\sum_{j=1}^\infty a_j\sin 2\pi jx $ such that $a_j= \mathcal O(j^{-s})$, $s>1/2$. 
\end{remark}\vskip 5 pt 
 \begin{remark} Estimates (i)-(iii) also provide sharp bounds to GCD sums indexed on FC sets. Estimate  (i) with
$s=\tau=1$ further implies
  \begin{eqnarray*} 
\sum_{k,\ell\in K}c_k   c_\ell
 \frac{(k,\ell)^{2}}{k\ell}  &\le&\frac{\pi^2}{6}  \Big(\sum_{k \in K}\frac{\p(k)}{k^2} \Big)\Big(\sum_{k \in K}
  |c_k|^2  \s_{ -1}(k) \Big), \end{eqnarray*}  
where $\p(n)$ is Euler totient function.  Indeed \begin{eqnarray*} \sum_{k,\ell\in K}c_k   c_\ell
  \frac{(k,\ell)^{2}}{k\ell}   &\le& \Big(\sum_{k \in K} \frac{1}{\s(k) } \Big)\Big(\sum_{k \in K}  |c_k|^2  \s_{ -1}(k) \Big) 
 \cr  &\le&\frac{\pi^2}{6}  \Big(\sum_{k \in K}\frac{\p(k)}{k^2} \Big)\Big(\sum_{k \in K}  |c_k|^2  \s_{ -1}(k) \Big),\end{eqnarray*}
since (\cite{CMS}, p.10) ${\s(n)\p(n)} >{6{n^2}}/{\pi^2}$.
Concerning the first factor, notice that
$$\sum_{k=1}^\infty \frac{\p(k)}{k^s}= \frac{\zeta(s-1)}{\zeta(s)} \qq \quad \hbox{for $s>2$.} $$\end{remark}
 \vskip 3pt Recall for later use  that by Gronwall's estimates,
(\cite{Gr} p.\ 119--122),
\begin{equation} \label{gron}   \limsup_{n\to \infty} \frac{\log
\big(\frac{\s_\a(n)}{n^\a}\big)}{\frac{(\log n)^{1-\a}}{\log\log n}} =\frac{1}{1-\a}, \qq (0<\a<1)
\end{equation}
 Further,  \begin{equation} \label{gron1}  \limsup_{n\to \infty}\frac{\s (n)}{n\log\log n}= e^\g,
 \end{equation}
   where $\g$ is Euler's constant.
(The validity   of the seemingly slightly stronger inequality $ \s(n)  < e^\g n\log\log n  $ for all $n\ge 5041$, is known to be
equivalent to the Riemann Hypothesis.)
 \begin{remark} One can substitute in the preceding estimates to    $\s_{-1}(n)$ the prime divisor function $\o(n)$ (counting the number
of prime divisors of $n$). By
 Duncan's inequality      $  \s_{-1}(n) < \frac{7 \o(n)+10}{6}$. 
  Further,  by  Satyanarayana and Vangipuram result, if $n$ is odd and $3\!\!\not| n$, $5\!\!\not|n$, then
 $\s_{-1}(n)<   (\frac{P^-(n)+2\o(n) +21/2}{P^-(n)+3} )^{1/2}$,  
where $P^-(n)$ is the smallest prime divisor of $n$. See \cite{CMS} p.78-79.
\end{remark}
\begin{remark} A strenghtened form of Theorem-(i) \ref{p1}, involving a more delicate analysis is proved in Section \ref{s7}.
\end{remark}
 \subsection{Eigenvalues arithmetical estimates.} The recent estimates  of the eigenvalues  of the
arithmetical matrix
$$ M(K,s)= \Big\{\frac{(k,\ell)^{2s}}{k^s\ell^s}\Big\}_{  k,\ell \in K} $$ established in \cite{H,ABS,BoS} are sharp but are not of
arithmetical type. An important and quite challenging question is precisely to know  whether it is possible to provide bounds of this
type, expressed in a simple way by arithmetical functions.  In this direction,  the following GCD sum estimate established in 
\cite{BW}, (Proposition 1.13) is relevant.  
\begin{proposition}\label{prop01e} Let $0<s\le 1$.   For any $k\in K$ (letting 
$K_-=\min{K}$, $K^+= \max(K)$),
\begin{eqnarray*} \sum_{\ell\in K\atop
\ell\not =k}  \frac{(k,\ell)^{2s}}{k^s\ell^s}\le \begin{cases}
2\big(\log \frac{K_+  } {K_- } \big)\s_{-1}(k)   &\quad {\rm if}\ s =1,
\cr  2^sk ^{s-1}\big(\int^{K_+  }_{K_-  }\frac{\dd
u}{u^s}\big)
\s_{ 1-2s}(k)   &\quad {\rm if}\ 0< s<1. \end{cases}
\end{eqnarray*}
 \end{proposition}
   It is easy to derive  eigenvalues estimates of $M(K,s)$ for  $K$  arbitrary. \begin{corollary}\label{cor01}  Let $0<s\le 1$. Let
$\lambda(k,s), k\in K$  be the eigenvalues of 
$M(K,s)$. Then for any
$k\in K$,
  \begin{eqnarray*}|\lambda(k,s)-\zeta(2s)|\le \begin{cases}
2\big(\log \frac{K_+  } {K_- } \big)\displaystyle{\sup_{k\in K}} \s_{-1}(k)   &\quad {\rm if}\ s =1,
\cr   2^s (\frac{K_+}{K_-})  ^{1-s}
\displaystyle{\sup_{k\in K}} \s_{ 1-2s}(k)   &\quad {\rm if}\ s<1. \end{cases}
\end{eqnarray*}
 \end{corollary}
 Gronwall's estimates (\ref{gron}) further allow to   provide quantitative bounds.  \begin{proof} We apply Ger\v sgorin's theorem stating that the eigenvalues of an
$n\times n$ matrix $(a_{i,j })$ with complex entries lie in the union of the closed disks (Ger\v sgorin disks)
\begin{equation}|z-a_{i,i}| \le \sum_{j=1\atop 
j\not =i}^n |a_{i,j }| \qq\quad  (i=1,2,\ldots , n) \label{Ger} 
  \end{equation} 
in the complex plane, see for instance \cite{Va}.
 Hence
\begin{equation*} |\lambda(k,s)-\zeta(2s)| \le \sup_{k\in K} \sum_{\ell\in K\atop 
\ell\not =k}  \frac{(k,\ell)^{2s}}{k^s\ell^s} . 
  \end{equation*} 
Applying  Proposition \ref{prop01e} and  noticing that when $s<1$,
$$ \sum_{\ell\in K\atop \ell\not =k}  \frac{(k,\ell)^{2s }}{k^s
\ell^s }\le   2^s (\frac{K_+}{K_-})  ^{1-s}
\s_{ 1-2s}(k) ,$$
allows to conclude.  \end{proof}
When combined with the classical weighted estimate for quadratic forms:

{\it  For any system of complex numbers $\{x_i\}$ and  $ \{\a_{i,j}\}$},
 \begin{eqnarray}\label{wqe} \Big|  \sum_{ 1\le i,j\le n\atop i\not= j}  x_ix_j\a_{i,j} \Big|\le  \frac1{2}\sum_{ i=1}^n |x_i|^2\Big(  \sum_{\ell
=1\atop
\ell\not = i }^n (|\a_{i,\ell} | +|\a_{
\ell ,i} |) \Big) ,
\end{eqnarray}
Proposition \ref{prop01e} immediately implies that
  \begin{eqnarray}\label{61}\qq    \big\|\sum_{k\in K} c_k f^s _k\big\|_2^2  \le\begin{cases}
  2^s (\frac{K_+}{K_-})  ^{1-s}  \sum_{k\in K} \s_{1-2s}(k) c_k^2&\quad \hbox{if $1/2< s\le 1$}\cr
 2\big(\log \frac{K_+  } {K_- } \big)\sum_{k\in K} d(k)c_k^2 &\quad \hbox{if $s=1/2$},  \end{cases}\end{eqnarray}
as observed in \cite{BW}. These estimates turn up to   be  of   crucial  use in the  last part of the proof of
Theorem \ref{t1a}.
\begin{remark} Let $p_1,\ldots, p_N$ be distinct prime numbers and let $G=\langle p_1,\ldots, p_N\rangle$ the associated  multiplicative semi-group. A simple consequence of (\ref{wqe}) is also that 
\begin{eqnarray*} \sum_{i,j=1}^Nx_i\overline{x}_j \frac{(n_{i},n_{j})^{2\a }}{n_{i}^\a
n_{j}^\a }  &\le  & C\prod_{i=1}^N\Big( \frac{1}{1- p_i^{-\a      
}}\Big)\, \sum_{i =1}^N|x_i|^2, 
\end{eqnarray*}
for any $\a>0$, any $n_1, \ldots, n_N\in G$ and any complex numbers $x_1, \ldots, x_N$. And the constant $C$ is absolute. \end{remark}
 
\subsection{An arithmetical $\boldsymbol \O$-theorem for the Zeta function.}

We  derive from a recent result of Hilberdink  \cite{H} (see Theorem 3.3) an arithmetical type $\Omega$-result for the Riemann Zeta
function which is, to our knewledge, the first of this kind in the theory.  \begin{theorem} \label{OFC1}Let $\s>1/2$. There exist 
a positive constant $c_\s$ depending on $\s$ only and a  positive absolute constant $C$, such that 
for any FC set
$K$ such that 
   $$K^+=
\max\{ k, k\in K\}
\le  T , \qq  \max_{ k,\ell \in K}\ \frac{k\vee \ell }{(k,\ell)}\ge  c_\s, $$
and any $0\le
\e <\s$,\begin{eqnarray*}
   (1+ C)\Big(\sum_{n\in K } \frac{1}{n^{ 2\e}}  \Big)  \ \max_{1\le t\le T}  |\zeta(\s + it)|^2
  &  \ge & \frac{\zeta(2\s)}{2}
 \sum_{n\in K }\frac{1}{n^{2\e}}\s_{-s+\e}( n)^2  
  \cr & &  - \frac{1}{T^{ 2\s-1}}  \Big(\sum_{  \ell \in K } \frac{1}{\ell^{ \e}} \Big) \Big(\sum_{  k \in K }      k ^{1-\s-\e }\log ( kT)  \Big\}.
\end{eqnarray*} 
   \end{theorem}

\begin{theorem}\label{OFC2}Let $\s>1/2$. For any   integer $\nu\ge 2$ such that $ \max_{ [k,\ell ]|\nu}\ \frac{ (k\vee
\ell) }{(k, \ell)}\ge c_\s$,   and $0\le \e <\s$, we have 
  \begin{eqnarray*}
    \max_{1\le t\le T}  |\zeta(\s + it)| 
   &  \ge & c \Big(
\frac{1}{ \s_{-2\e} (\nu)}\sum_{n|\nu}\frac{\s_{-s+\e}( n)^2}{n^{2\e}} \Big)^{1/2}
    ,
\end{eqnarray*}
whenever $\nu$ and $T$ are such that
$$ \frac{\s_{-\e} (\nu)\s_{1-\s-\e} (\nu) \log ( \nu T)}{  \sum_{n|\nu}\frac{\s_{-s+\e}(
n)^2}{n^{2\e}}  }\le \frac{\zeta(2\s)^{1/2}}{4}  T^{(2\s-1)}.   $$
Here  $c =\frac{\zeta(2\s)}{2\sqrt{1+C}   }$ and $c_\s, C$ are as in Theorem \ref{OFC1}.   \end{theorem}
 By taking $\nu$ a product of primes, it is easy to recover Hilberdink's result from Theorem \ref{OFC2}.

\section{Almost Everywhere Convergence Results.}
   We first
apply Theorem
\ref{t1} to almost everywhere convergence. We obtain new convergence conditions of mixed type, namely multipliers  partly expressed by arithmetical
functions. We will prove 
\begin{theorem}\label{c1a}   Assume that $a_m= \mathcal O(m^{-\a} )$ for some $\a>1/2$.
\vskip1 pt \noi i) Let  $1/2<\a<1$. Then the series $ \sum_{k\ge 1} c_k f_k$ converges almost everywhere whenever  the following
condition is satisfied,
\begin{eqnarray*}  \sum_{ k\ge 3} c_k^2 (\log k)^{4(1-\a)}({\log\log k})^{2(1-\a)} d(k^2)<\infty .   
      \end{eqnarray*}
 
\vskip1 pt \noi ii) Let $\a=1$. Then the same conclusion holds true  if  the above condition  is replaced by 
\begin{eqnarray*}  \sum_{ k\ge 3} c_k^2  d(k^2) (\log\log k)^2   <\infty .   
      \end{eqnarray*}

\vskip1 pt \noi iii) Assume that $a_m= \mathcal O(m^{-1/2}(\log m)^{-(1+h)/2} )$ for some $h>1$. Then the same conclusion holds true  under
the following condition
\begin{eqnarray*}\sum_{k\ge 3}  c_k^2d(k^2)   (\log k)^2  (\log\log k)^{1-h}
  <\infty .  
 \end{eqnarray*}
\end{theorem} 
  These  arithmetical conditions are   meaningful for coefficient sequences supported by sets of integers $k$
having few divisors. 
In      \cite{BW} Theorem 2.8, we showed that the condition
 \begin{eqnarray*} 
 \sum_{ k\ge 1} c_k^2(\log k)^2 \s_{1-2\a}({ k} ) <\infty,
\end{eqnarray*}
also implies the  convergence almost everywhere of the series $ \sum_{k\ge 1} c_k f_k$. Although not exactly comparable with the condition given in $i$),  this one
yields a better condition for coefficient sequences supported by   integers with few divisors. A similar remark holds
  concerning the general condition given in  \cite{BW} (see  Corollary 2.6 and Remark 2.7). The condition given in (ii) has to be compared
with the one in Theorems \ref{t1a}, \ref{t1ab}.
  
   As to  (ii) and (iii),  the non-arithmetical factors of the multipliers are significantly better  than those in Theorem \ref{t1a}, and Theorem 1.1 in \cite{We},  respectively. Recall concerning (ii)  
that condition (see Theorems 3,7 in \cite{ABS})
   \begin{eqnarray*}  \sum_{ k\ge 3} c_k^\g  (\log\log k)^\g   <\infty ,  
      \end{eqnarray*}
for  $\g<2$ is necessary for the  convergence almost everywhere of the series $ \sum_{k\ge 1} c_k f_k$.
\vskip 3 pt
We further prove the following   almost everywhere convergence result concerning the Banach space ${\rm BV}(\T)$ of functions with bounded variation.    
 \begin{theorem}\label{t1a} Let $f\in {\rm BV}(\T)$,
$\langle f,1\rangle=0$. Assume that 
\begin{eqnarray} \label{coeff2}  \sum_{k\ge 3}  c_k^2\frac{(\log\log k)^4}{(\log\log \log k)^2}  <\infty . 
\end{eqnarray}
Then the series $ \sum_{k } c_kf_k$ converges almost everywhere. 
\end{theorem}
This slightly improves
 Theorem 3 in \cite{ABS} ($n_k=k$), where it was assumed that the series
\begin{eqnarray} \label{abs4} 
    \sum_{k=1}^\infty c_k^2 \, (\log\log k)^\g     
\end{eqnarray}
converges for some $\g>4$.    
\vskip 3 pt
 We will also prove the following rather delicate result where multipliers have arithmetical factors. \begin{theorem}\label{t1ab} Let $f\in {\rm BV}(\T)$,
$\langle f,1\rangle=0$. Assume that for some real $b>0$, 
\begin{eqnarray} \label{coeff1}
 \sum_{  k \ge 3}
   c_k^2  (\log\log k)^{2 + b} \s_{ -1+\frac{1}{(\log\log k)^{  b/3}}  }(k)      <\infty . 
\end{eqnarray}
Then the series $ \sum_{k } c_kf_k$ converges almost everywhere. 
\end{theorem}
 We will derive it, as well as (\ref{coeff2}), directly from   Theorem \ref{p1}, thus    without
using analysis on the polydisc as in \cite{ABS}.    
  \begin{remark}In spite of the regular decay of its Fourier coefficients, it is well-known that a function $f\in {\rm BV}(\T)$ may have very pathological
behavior.
 Jordan \cite{Jo} gave in 1881 a remarkably simple and elegant   construction of a function  with bounded variation, having positive jumps on each
rational, and being continuous almost everywhere.
 \end{remark} \begin{remark}Theorem \ref{t1a} applies to the case $s=1$  in (\ref{funct}) which corresponds to the Fourier expansion of   the function $
  \langle  x\rangle= \frac{\pi}{2}(1-2  x)$, $0\le x\le 1$  
\begin{equation}  \langle 
x\rangle= \sum_{1}^\infty
\frac{\sin  2\pi n x}{n}\qq (0< x< 1) 
\end{equation} 
the series being discontinuous at $x=0$.  It is quite interesting to notice by expanding $\langle  x\rangle $    with respect to the  system
$\cos (n+\frac{1}{2})x$, $\sin (n+\frac{1}{2})x$, $n=0,1, \ldots$ (which is orthogonal and complete over any interval of length $2\pi$),
that one also gets (\cite{Z} p.71)
\begin{equation} \langle  x\rangle  =\frac{4}{\pi}\sum_{0}^\infty
\frac{\cos 2\pi (n+\frac{1}{2})x}{(2n+1)^2}\qq (0\le x\le 1),
\end{equation}
where this time the series is absolutely and uniformly  convergent. Let $\varsigma(x) $ denote the series in the right handside.    Further, it is
not a complicated task to prove that the series $\sum   c_k\varsigma(n_k x)  $ converges for almost every
$x$ under the minimal condition $\sum   c_k^2<\infty$. However $\varsigma(x) $ is
$2$-periodic whereas $\langle  x\rangle$ is
$1$-periodic. The study of the system $ \{\langle 
nx\rangle, n\in \N\}$ goes back to Riemann's work \cite{R}. Davenport \cite{D1, D2} much investigated its properties.  It is known  
that this system
  possesses smoothness properties going at the opposite of those of the trigonometrical system (the series
$ 
\sum_k c_k  \langle  kx\rangle$ is never continuous unless the  coefficients $ c_k
$ all vanish). 
 We refer to  Jaffard 
 \cite{J}.
 However, the a.s. convergence properties of series attached to this system seem to remain relatively close to those of the
trigonometrical system, namely to belong close to the domain of applicability of Carleson's theorem.

\end{remark} 
 
 \subsection{A complement to Wintner's Theorem.}
We finally also prove an important  complementary result to Wintner's 
famous characterization of mean convergence of series $\sum_{k=0}^\infty c_k f_k $. Recall some necessary facts.    Let  
$f\in L^2(\T)$ with
$\langle f, 1\rangle=0$ 
 and denote $ \bar {f} =\{f_n, n\ge 0\}$ where we recall that 
$f_n(x)= f(nx)$. We say that the system
$\bar {f} $ is mean convergent if the series  
$\sum_{k=0}^\infty c_k f_k $  converges in
  $L^2(\T)$ for any $ \{c_k, n\ge 0\}  \in \ell^2$.  
This property is characterized by the following well-known theorem. 
\begin{theorem}[Wintner \cite{Wi}]   \label{win}
\vskip 3pt The following statements are equivalent:
\vskip 3pt \noi 1.   The series   $\sum_{k=1}^\infty c_k f_k( x)$  converges in
  $L^2(\T)$ for any coefficient sequence $(c_k)_k\in \ell^2$.
\vskip 3pt \noi 2. There exists a constant $c>0$ such that for any $n\ge 1$ and
any reals $\{c_k, 1\le k \le n\}$ we have
 \begin{equation}\label{qos} \big\|\sum_{k=1}^n  c_kf_k \big\|^2\le c\sum_{k=1}^nc_k^2  .
\end{equation}
\vskip 3pt \noi 3. The infinite matrix
 $\big(\langle f_k, f_\ell\rangle \big)_{k,\ell}  $
defines a bounded operator on $\ell^2$.
\vskip 3pt \noi 4. The Dirichlet series
 $\sum_{n=1}^\infty a_n n^{-s}$ and 
$\sum_{n=1}^\infty b_n n^{-s}$ are regular and bounded
in the half-plane ${\Re (s)}>0$.
\end{theorem} 
   Suppose $\bar {f}$ is mean convergent. It is natural to ask whether there always exists a  
class of    coefficients $(c_k)_k$ for which  the series  $\sum_{k=1}^\infty c_k\p_k $ will  converge  almost everywhere. The   theorem below answers this
affirmatively by identifying  a general class of    coefficients.
\vskip 2 pt      
Recall a useful notion. A sequence of coefficients
$ \{c_k, n\ge 0\}$ is called   {\it universal} if for any orthonormal system $\Phi $ of functions defined on
a bounded interval (and possibly extended periodically over the real line), the
 series $\sum_{k=1}^\infty c_k\p_k $ converges a.e.  
  
 \begin{theorem}\label{sw} Assume that 
 $\bar {f}$ is mean convergent. Then the series $\sum_{k=1}^\infty c_k f_k( x)$ converges a.\!\! e. for any {\it universal}
coefficient sequence $(c_k)_k$.
\end{theorem} 
  
\vskip 3pt The paper is organized as follows: in Sections \ref{s2}  to \ref{s6} we prove
the results stated before. In Section \ref{s7}, a strenghtened form of Theorem \ref{p1} is proved. 

\section{\bf Proof of Theorem \ref{t1}}\label{s2} Let $\d$ be the arithmetical function defined by
\begin{eqnarray} \label{delta}   \d(n)=\begin{cases}1\quad & {\rm if}\ n=1,\cr
0\quad & {\rm if}\ n \not=1. \end{cases} \end{eqnarray} Let $\m$ denotes the M\"obius  function and recall that 
\begin{eqnarray} \label{sp} \sum_{d|n}\m(d)=\d(n) . \end{eqnarray} 
We have 
$$\big\|\sum_{k\in K} c_k f_k\big\|_2^2 =  \sum_{k, \ell\in
K}c_kc_\ell\sum_{\nu=1}^\infty a_{\frac{\nu k}{(k,\ell)}}a_{\frac{\nu
\ell}{(k,\ell)}}. $$
 We  decompose the right-hand side according to the values taken by $(k,\ell)$,  
 \begin{eqnarray}\label{sd0} 
 \sum_{k, \ell\in K}c_kc_\ell\sum_{\nu=1}^\infty a_{\frac{\nu k}{(k,\ell)}}a_{\frac{\nu \ell}{(k,\ell)}}&=& 
\sum_{d\in F(K)} S_d,
\end{eqnarray} 
where 
\begin{eqnarray}\label{sd}S_d= \sum_{k, \ell\in K\atop (k,\ell)=d}c_kc_\ell\sum_{\nu=1}^\infty a_{\frac{\nu k}{d}}a_{\frac{\nu
\ell}{d}} . 
\end{eqnarray}   
  
   We claim that
\begin{eqnarray}\label{s1} |S_1 |  &\le & \Big(\sum_{  k \in K }|c_k|^2\psi(k)\Big) \Big(\sum_{  k \in K }\frac{A_{  k  }}{ \psi(k)
} \theta(k) \Big) .   
   \end{eqnarray}
  Indeed, by (\ref{delta}), (\ref{sp}),
 \begin{eqnarray*} S_1 &=& \sum_{k, \ell\in K }c_kc_\ell\sum_{\nu=1}^\infty a_{ {\nu k} }a_{ {\nu
\ell} }\d((k,\ell)) =\sum_{k, \ell\in K }c_kc_\ell\sum_{\nu=1}^\infty a_{ {\nu k} }a_{ {\nu
\ell} }\sum_{d|(k,\ell)} \m(d)\cr 
& =& \sum_{\nu=1}^\infty\sum_{d\in F(K)}\m(d)\sum_{   k,\ell\in K \atop d|k, d|\ell}c_kc_\ell  a_{ {\nu k} }a_{ {\nu
\ell} }
= \sum_{\nu=1}^\infty\sum_{d\in F(K)}\m(d)\Big(\sum_{    k \in K \atop d|k}c_k  a_{ {\nu k} } 
 \Big)^2 .   
  \end{eqnarray*}
This thus factorizes, and now we can apply Cauchy-Schwarz's inequality to get
\begin{eqnarray*} |S_1|  & \le &\sum_{\nu=1}^\infty\sum_{d\in S\hskip -1,9ptF(K)} \Big(\sum_{  k \in K \atop d|k } |c_k|\sqrt{\psi(k)}. 
  \frac{a_{ {\nu k}}}{\sqrt{\psi(k)}}   \Big)^2
 \cr  & \le  & \Big(\sum_{  k \in K }|c_k|^2\psi(k)\Big)\sum_{\nu=1}^\infty \sum_{d\in S\hskip -1,9ptF(K)} 
   \sum_{k\in K   \atop d|k     }   \frac{a^2_{ {\nu k}
}}{ \psi(k) } 
\cr & = &   \Big(\sum_{  k \in K }|c_k|^2\psi(k)\Big) \Big(\sum_{  k \in K }\frac{A_{  k  }}{ \psi(k)
} \sum_{   d\in S\hskip -1,9ptF(K) \atop d|k  }   1  \Big)    \cr  & =&   \Big(\sum_{  k \in K }|c_k|^2\psi(k)\Big) \Big(\sum_{  k \in K
}\frac{A_{  k  }}{ \psi(k) }  \theta(k) \Big),    
  \end{eqnarray*}

Now let $K_d= \frac{1}{d}(d\N\cap K) $. For the sum   $S_d$
defined in (\ref{sd}), we have
 \begin{eqnarray*}S_d&= &\sum_{  k, \ell\in K\atop (k,\ell)=d}c_kc_\ell\sum_{\nu=1}^\infty a_{\frac{\nu k}{d}}a_{\frac{\nu
\ell}{d}}  
= \sum_{  k', \ell'\in K_d\atop (k',\ell')=1}c_{k'd}c_{\ell'd}\sum_{\nu=1}^\infty a_{ {\nu k' } }a_{ {\nu
k' } } .
   \end{eqnarray*}  
Thus $S_d $ has just same form than the sum $S_1$ studied before, with $K_d, k',c_{k'd}, a_{k'd}$ in place of $K, k,c_{k}, a_{k}$.  We deduce from   (\ref{s1}) that
 \begin{eqnarray}\label{s1a} S_d   &\le & \Big(\sum_{k'\in K_d}   |c_{k'd}|^2\psi(k' )\Big)
\Big(\sum_{k'\in K_d} 
\frac{A_{  k'  }}{
\psi(k')}\theta(k')
\Big)
\cr &= & \Big(\sum_{k \in K\atop d|k}   |c_{k}|^2\psi(\frac{k}{d})\Big)
   \Big(\sum_{k \in K\atop d|k} 
\frac{A_{  \frac{k}{d}  }}{
\psi(\frac{k}{d})} \theta(\frac{k}{d}) 
\Big)  .
   \end{eqnarray}
   Using (\ref{sd0}), we get 
 \begin{eqnarray}\label{s1b} \big\|\sum_{k\in K} c_k f_k\big\|_2^2&=&\sum_{d\in F(K)}  S_d   \le  \sum_{d\in F(K)}\Big(\sum_{k \in
K\atop d|k}    |c_{k}|^2\psi(\frac{k}{d})\Big)
   \Big(\sum_{k \in K\atop d|k} 
\frac{A_{  \frac{k}{d}  }}{
\psi(\frac{k}{d})} \theta(\frac{k}{d}) 
\Big)  
\cr  &\le &\Big(\sup_{d\in F(K)} \sum_{k \in K\atop d|k} 
\frac{A_{  \frac{k}{d}  }}{
\psi(\frac{k}{d})} \theta(\frac{k}{d})\Big)   \sum_{k \in K }    |c_{k}|^2\sum_{d\in F(K)\atop d|k}
\psi(\frac{k}{d}) 
     \cr &= & B \sum_{k \in K }  |c_{k}|^2 \psi*\zeta_0(k).
    \end{eqnarray}
 

  \section{\bf Proof of Theorem \ref{p1}}  \label{s3}
 Let $\{c_k, k\ge 1\}$ be a sequence of coefficients
$\{c_k, k\ge 1\}$ supported by
$K$,    
 ($c_k=0$ if $k\notin K$). Let $\e>0$. We recall that the generalized Euler totient function
$J_\e$ is the   multiplicative arithmetical function defined by 
$$J_\e(n)= \zeta_\e *\m(n)=\sum_{d|n} d^\e \m(\frac{n}{d}).$$
By M\"obius inversion Theorem,  
\begin{eqnarray} \label{m} n^\e =\sum_{d|n} J_\e (d). \end{eqnarray}
\ Step (1) is as in  \cite{BW1}, except that we introduce 
   an   arithmetic function $\psi
$. It is necessary to display it here. Step (ii) uses basic properties of   Dirichlet convolutions.  
\vskip 3 pt \noi 
(1)  Noticing that
if  
$d|k$  and
$k\in K$,  then
$d\in F(K)$, we have  by (\ref{m}) 
$$ (k,\ell)^\e =\sum_{d\in F(K)}  J_\e (d)  {\bf 1}_{d|k} {\bf 1}_{d|\ell}.  $$
Thus
\begin{eqnarray}  \label{HS1a}
L\ :=\  \sum_{k,\ell=1}^n
 c_k   c_\ell\frac{(k,\ell)^{2s}}{k^s\ell^s}&=& \sum_{k,\ell \in K}  \frac{c_k   c_\ell }{k^s\ell^s}\Big\{\sum_{d\in F(K)} 
J_{2s} (d)  {\bf 1}_{d|k} {\bf 1}_{d|\ell}\Big\}  
. 
 \end{eqnarray}
Writing $k=ud$, $\ell=vd$ and noting that $u,v\in F(K)$, we have 
\begin{eqnarray*} 
L  &\le& \sum_{u,v\in F(K)} \frac{1}{u^sv^s} \Big(\sum_{d\in F(K)} \frac{J_{2s}
(d)}{d^{2s}}c_{ud}c_{vd}   \Big) . 
 \end{eqnarray*}
By the Cauchy-Schwarz inequality,
$$ \sum_{d\in F(K)} \frac{J_{2s}
(d)}{d^{2s}}c_{ud}c_{vd}    \le \Big(\sum_{d\in F(K)} \frac{J_{2s}
(d)}{d^{2s}}c_{ud}^2  \Big)^{1/2}\Big(\sum_{d\in F(K)} \frac{J_{2s}
(d)}{d^{2s}} c_{vd}^2   \Big)^{1/2}.$$
Hence,
\begin{eqnarray*} 
L  &\le& \Big[\sum_{u \in F(K)} \frac{1}{u^s } \Big(\sum_{d\in F(K)} \frac{J_{2s}
(d)}{d^{2s}}c_{ud}^2  \Big)^{1/2}\Big]^2 . 
 \end{eqnarray*}
Let   $\psi $ be a positive arithmetic function. Writing $\frac{1}{u^s}=\frac{1}{u^{s/2} \psi(u)^{1/2}} \frac{\psi(u)^{1/2}}{u^{s/2}}$ and applying  Cauchy-Schwarz's inequality again gives,
\begin{eqnarray*} 
L &\le& \Big(\sum_{u \in F(K)} \frac{1}{u^s\psi(u) } \Big)\Big(\sum_{u \in F(K)}
\frac{\psi(u)}{u^s  } \sum_{d\in F(K)}
\frac{J_{2s} (d)}{d^{2s}}c_{ud}^2  \Big)  . 
 \end{eqnarray*}
Let $F^2(K)=\{ ud: u,d\in F(K)\}$. Then
\begin{eqnarray} \label{depp}
\sum_{u \in F(K)}
\frac{\psi(u)}{u^s  } \sum_{d\in F(K)}
\frac{J_{2s} (d)}{d^{2s}}c_{ud}^2    &\le&  \sum_{\nu \in F^2(K)}  c_\nu^2  \sum_{u \in
F(K)\atop u|\nu }
\frac{\psi(u)}{u^s  } 
\frac{J_{2s} (\frac{\nu}{u   })}{(\frac{\nu}{u   })^{2s}}  
\cr &=&  \sum_{\nu \in K} \frac{ c_\nu^2}{  \nu^{2s}   }   \sum_{u \in
F(K)\atop u|\nu }
  J_{2s}\big( \frac{\nu}{u   }\big) u^{ s} \psi(u)    , 
 \end{eqnarray}
since $c_\nu=0$ if $\nu\notin K$. Hence we get
\begin{eqnarray} \label{dep0}
L &\le& \Big(\sum_{u \in F(K)} \frac{1}{u^s\psi(u) } \Big)\Big(\sum_{\nu \in K} \frac{ c_\nu^2}{  \nu^{2s}   }   \sum_{u \in
F(K)\atop u|\nu }
  J_{2s}\big( \frac{\nu}{u   }\big) u^{ s} \psi(u) \Big)  . 
 \end{eqnarray}

\vskip 3 pt \noi (2)  
Choose $\psi(u) = u^{-s} \psi_1(u)\s_\tau(u)$ (recalling that $\psi_1(u)>0$  is non-decreasing). Then,
\begin{eqnarray*}
L &\le& \Big(\sum_{u \in F(K)} \frac{1}{\psi_1(u)\s_\tau(u) } \Big)\Big(\sum_{\nu \in K} \frac{ c_\nu^2}{  \nu^{2s}   }   \sum_{u \in
F(K)\atop u|\nu }
  J_{2s}\big( \frac{\nu}{u   }\big) \psi_1(u)\s_\tau(u) \Big) 
  \cr &\le& \Big(\sum_{u \in F(K)} \frac{1}{\psi_1(u)\s_\tau(u) } \Big)\Big(\sum_{\nu \in K} \frac{ c_\nu^2 \psi_1(\nu)
}{  \nu^{2s}   }   \sum_{u \in
F(K)\atop u|\nu }   J_{2s}\big( \frac{\nu}{u   }\big)\s_\tau(u) \Big)  . 
 \end{eqnarray*}
As $\nu \in K$,
\begin{eqnarray*} 
      \sum_{u \in
F(K)\atop u|\nu } 
  J_{2s}\big( \frac{\nu}{u   }\big) \s_\tau(u)=J_{2s}* \s_\tau(\nu) . 
 \end{eqnarray*}
    By commutativity and associativity of the Dirichlet convolution,
\begin{eqnarray*} 
     J_{2s}* \s_\tau =     (\zeta_{2s}*\m)* (\zeta_\tau*\zeta_0)=     (\zeta_{2s}*\zeta_\tau)* ( \zeta_0*\m)=    
(\zeta_{2s}*\zeta_\tau)*\d , 
 \end{eqnarray*}
since by (\ref{sp}), $\zeta_0*\m=\d$. 
Further
\begin{eqnarray*} 
      \zeta_{2s}*\zeta_\tau (n)= \sum_{d|n} d^{2s}\big(\frac{n }{ d  }\big)^\tau=n^\tau \sum_{d|n} d^{2s-\tau} =n^\tau \s_{2s-\tau}(n).
 \end{eqnarray*}
 Consequently,
\begin{eqnarray*} 
      J_{2s}* \s_\tau(\nu )&=& \sum_{n|\nu}n^\tau \s_{2s-\tau}(n) \, \d\big(\frac{\nu }{ n  }\big)=\nu^\tau \s_{2s-\tau}(\nu) .
 \end{eqnarray*}

By reporting
\begin{eqnarray} \label{oa}
L  &\le& \Big(\sum_{u \in F(K)} \frac{1}{\psi_1(u)\s_\tau(u) } \Big)\Big(\sum_{\nu \in K}
\frac{ c_\nu^2\psi_1(\nu)}{  \nu^{2s}   }\nu^\tau \s_{2s-\tau}(\nu) \Big) . 
\cr 
&=& \Big(\sum_{u \in F(K)} \frac{1}{\psi_1(u)\s_\tau(u) } \Big)\Big(\sum_{\nu \in K}
  c_\nu^2  \psi_1(\nu)\s_{ \tau-2s}(\nu) \Big), 
 \end{eqnarray}
as claimed. 
Taking $\psi_1(u)\equiv 1$ gives \begin{eqnarray*} 
L
&\le& \Big(\sum_{u \in F(K)} \frac{1}{\s_\tau(u) } \Big)\Big(\sum_{\nu \in K}
  c_\nu^2  \s_{ \tau-2s}(\nu) \Big), 
 \end{eqnarray*}
which is (i). Taking now $\psi_1(u)= \overline{\s}_{-s}(u)$, $\tau =s$ gives
\begin{eqnarray*} 
L  &\le &  \Big(\sum_{u \in F(K)} \frac{1}{\overline{\s}_{-s}(u)\s_s(u) } \Big)\Big(\sum_{\nu \in K}
  c_\nu^2 \overline{\s}_{-s}(u)\s_{-s}(u) \Big), 
 \end{eqnarray*}
 namely (ii). Finally  let $s=1=\tau$ and $\psi_1(u)=\log\log u$. Then, by (\ref{oa}) again,\begin{eqnarray*} 
L
&\le& \Big(\sum_{u \in F(K)} \frac{1}{\s(u)\log\log u  } \Big)\Big(\sum_{\nu \in K}
  c_\nu^2 \s_{-1}(\nu)\log\log \nu \big)  \Big)
  \cr&\le& \frac{\pi^2}{6 } \Big(\sum_{u \in F(K)} \frac{\p(u)}{{u^2}\log\log u } \Big)\Big(\sum_{\nu \in K}
  c_\nu^2 \s_{-1}(\nu) \log\log \nu \Big),
 \end{eqnarray*}
since (\cite{CMS}, p.10) ${\s(n)\p(n)} >{6{n^2}}/{\pi^2}$.  This is  (iii) and the proof is now complete.  \vskip 3pt 
 
\begin{remark}
 Quite similarly, one can also prove that \begin{eqnarray*} 
L  &\le&  \Big(\sum_{u \in F(K)} \frac{1}{\sum_{d|\nu } d\,  \log\log d} \Big)\Big(\sum_{\nu \in K}
  c_\nu^2 \sum_{d|\nu } \frac{  \log\log d}{  d   } \Big) . 
 \end{eqnarray*}
 \end{remark} 
\section{Proof of Theorems \ref{OFC1}, \ref{OFC2}.}
 We begin with a lemma and a corollary which are slightly improving on
Proposition 3.3 in Hilberdink
 \cite{H}. We displayed with care the necessary calculations.
\begin{lemma}\label{zeta1}Let   $\s>1/2$. There exists a constant $C_{\s}$ depending on $\s$ only such that for any   positive integers $k,\ell$ and any real
$T$, 
\begin{eqnarray*} 
\int_1^T |\zeta(\s + it)|^2 \Big(\frac{k}{\ell}\Big)^{it} \dd t=\begin{cases}  T\zeta(2\s) \frac{(k,\ell)^{2\s}}{(k\ell)^\s} + H_1& \quad {\it if}\ T\ge \frac{k\vee \ell}{(k,\ell)}\cr &\cr 
H_2  & \quad {\it if}\ T< \frac{k\vee \ell}{(k,\ell)},\end{cases}
\end{eqnarray*}
 where
\begin{eqnarray*}|H_1| & \le&  C_{\s} \Big\{T\,\frac{(k,\ell)^{2\s}}{(k\ell)^\s}\Big( \frac{T}{\frac{k\vee \ell }{(k,\ell)} }\Big)^{1-2\s} + 
T^{2(1-\s)}\big\{1 + k^{1-\s}\log (kT)+ \ell ^{1-\s}\log (\ell T)\big\} \Big\}
\cr |H_2| & \le&  C_\s T^{2(1-\s)}\big\{1 + k^{1-\s}\log (kT)+ \ell ^{1-\s}\log (\ell T)\big\} .\end{eqnarray*}
\end{lemma}
\begin{proof}Recall   the   basic approximation result of   the Riemann
Zeta-function
(\cite{Te}, Theorem 3.5). 
Let   $\s_0>0$, $0<\d<1$. Then, uniformly for $\s\ge \s_0$, $x\ge 1$ , $0<|t| \le (1-\d)2\pi x$,
\begin{equation}\label{approx} \zeta(s) =\sum_{n\le x} {1\over n^s} -{x^{1-s}\over 1-s} +{\mathcal O}(
x^{-\s}).
\end{equation}

By taking $\d$ such that $2\pi(1-\d) =1$, $|t|=x$, and observing that $\big|\frac{t^{1-\s -it}}{1-\s- it}\big| \le |t|^{-\s}$, we deduce  
$$\sup_{|t|\ge 1, \s\ge \s_0 }|t|^{\s}\Big|\zeta(\s + it)-\sum_{k=1}^{|t|}\frac{1}{k^{\s + it} }\Big| <\infty.$$
And further, 
$$\sup_{|t|\ge 1, \s\ge \s_0 }|t|^{2\s-1}\Big||\zeta(\s + it)|^2-\big|\sum_{k=1}^{|t|}\frac{1}{k^{\s + it} }\big|^2\Big| <\infty.$$
 
\vskip 2 pt
We choose throughout $\s=\s_0>1/2$. We thus have
\begin{eqnarray}\label{dep}
\int_1^T |\zeta(\s + it)|^2 \Big(\frac{k}{\ell}\Big)^{it} \dd t&=&
\int_1^T \big|\sum_{k=1}^{|t|}\frac{1}{k^{\s + it} }\big|^2 \Big(\frac{k}{\ell}\Big)^{it} \dd t + H,
\end{eqnarray}
where $|H|\le C_{\s } T^{2(1-\s)}$. And
\begin{eqnarray*}
\int_1^T \big|\sum_{1\le n\le t}\frac{1}{n^{\s + it} }\big|^2 \Big(\frac{k}{\ell}\Big)^{it} \dd t &=&\sum_{1\le n,m\le T}\int_{n\vee m}^T \frac{1}{(nm)^\s}\Big(\frac{km}{\ell n}\Big)^{it} \dd t.
\end{eqnarray*}
$\bullet$  First assume that $(k,\ell) =1$. The solutions of the equation $km=\ell n$ are  $m= r\ell $, $n= r k$,  and the condition $n\vee m
\le T$ requires that $1\le r\le T/ (k\vee \ell)$. Hence,
\vskip 2 pt --- If $k,\ell $ are such that $(k\vee \ell)>T$, there is no solution  and thus, $$ \int_1^T \big|\sum_{1\le n\le t}\frac{1}{n^{\s + it}
}\big|^2
\Big(\frac{k}{\ell}\Big)^{it} \dd t =\sum_{1\le n,m\le T\atop km\not=\ell n}\int_{n\vee m}^T \frac{1}{(nm)^\s}\Big(\frac{km}{\ell n}\Big)^{it} \dd t.$$
\vskip 2 pt --- If $k,\ell $ are such that $(k\vee \ell)\le T$, then 
\begin{eqnarray*}
\sum_{1\le n,m\le T\atop km=\ell n}\int_{n\vee m}^T \frac{1}{(nm)^\s}\Big(\frac{km}{\ell n}\Big)^{it} \dd t&=& \frac{1}{(k\ell)^\s}\sum_{1\le r\le \frac{T}{k\vee \ell}}\frac{T-r(k\vee \ell)}{r^{2\s}}\ =\ \frac{T\zeta(2\s)}{(k\ell)^\s} + H,
 \end{eqnarray*}
where 
$$|H|\le C_\s \frac{1}{(k\ell)^\s}\frac{T^{2(1-\s)}}{(k\vee \ell)^{1-2\s}} .$$
$\bullet$ 
Now if $d:=(k,\ell) >1$, writing $k= k' d$, $\ell = \ell'd$, where $(k', \ell')= 1$, our initial integral reduces to 
 $$\int_{n\vee m}^T \frac{1}{(nm)^\s}\Big(\frac{k'm}{\ell' n}\Big)^{it} \dd t.$$
And by what preceeds, 
\vskip 2 pt (i) If $k,\ell $ are such that $\frac{k\vee \ell}{(k,\ell)}>T$, then 
$$ \int_1^T \big|\sum_{1\le n\le t}\frac{1}{n^{\s
+ it} }\big|^2 \Big(\frac{k}{\ell}\Big)^{it} \dd t =\sum_{1\le n,m\le T\atop km\not=\ell n}\int_{n\vee m}^T \frac{1}{(nm)^\s}\Big(\frac{km}{\ell
n}\Big)^{it}
\dd t.$$
\vskip 2 pt (ii) If $k,\ell $ are such that $\frac{k\vee \ell}{(k,\ell)}\le T$, then
\begin{eqnarray*}
\sum_{1\le n,m\le T\atop km=\ell n}\int_{n\vee m}^T \frac{1}{(nm)^\s}\Big(\frac{km}{\ell n}\Big)^{it} \dd t&=&\sum_{1\le n,m\le T\atop k'm=\ell' n}\int_{n\vee m}^T \frac{1}{(nm)^\s}\Big(\frac{k'm}{\ell' n}\Big)^{it} \dd t
\cr &=&   \zeta(2\s) T\,\frac{(k,\ell)^{2\s}}{(k\ell)^\s} + H,
\end{eqnarray*}
where 
$$|H|\le C_\s \frac{1}{(k'\ell')^\s}\frac{T^{2(1-\s)}}{(k'\vee\ell')^{1-2\s}}= C_\s T\,\frac{(k,\ell)^{2\s}}{(k\ell)^\s}
\Big( \frac{T}{\frac{k\vee \ell }{(k,\ell)} }\Big)^{1-2\s}
 .$$
\vskip 5pt 
Now consider the contribution provided by the indices $n,m$ such that $km\not= \ell n$, \begin{eqnarray*}
\Big|\sum_{1\le n,m\le T\atop km\not=\ell n}\int_{n\vee m}^T \frac{1}{(nm)^\s}\Big(\frac{km}{\ell n}\Big)^{it} \dd t\Big|&=&\Big| \sum_{1\le n,m\le T\atop km\not=\ell n}
\frac{1}{(nm)^\s}\, \frac{e^{iT \log \frac{km}{\ell n}}-e^{i(n\vee m) \log \frac{km}{\ell n}}}{i \log \frac{km}{\ell n}}\Big|
\cr &\le &\sum_{1\le n,m\le T\atop km\not=\ell n}
\frac{1}{(nm)^\s} \frac{1}{ |\log \frac{km}{\ell n}|}.\end{eqnarray*}
We operate as in the proof of Lemma 7.2 in \cite{Ti}. We have
\begin{eqnarray*}
\sum_{1\le n,m\le T\atop \ell n<km}
\frac{1}{(nm)^\s} \frac{1}{ |\log \frac{km}{\ell n}|}&=&\Big(\sum_{1\le n,m\le T\atop \ell n<\frac{1}{2}km}+\sum_{1\le n,m\le T\atop \frac{1}{2}km\le \ell n<km}\Big)
\frac{1}{(nm)^\s} \frac{1}{ |\log \frac{km}{\ell n}|} \cr &:= & S_1 + S_2.
\end{eqnarray*}
We have 
$$ S_1 \le \frac{1}{\log 2} \Big(\sum_{1\le n\le T}
\frac{1}{n^\s} \Big)^2\le C_\s T^{2(1-\s)}.$$
Now if $\frac{1}{2}km\le \ell n<km$, we can write $\ell n=km-r$ where $1\le r< \frac{1}{2}km$ and we note that $\log \frac{km}{\ell n}=-\log
\big(1-\frac{r}{km})>\frac{r}{km}$. Then
\begin{eqnarray*} S_2&\le& \sum_{m\le T}\sum_{1\le r< \frac{1}{2}km}\frac{1}{(km-r)^{\s}m^{\s}(\frac{r}{km})}\ \le \ 2^\s k^{1-\s}\sum_{m\le
T}m^{1-2\s}\sum_{1\le r< \frac{1}{2}km}\frac{1}{r}\cr &\le &C_\s  k^{1-\s}\sum_{m\le T}m^{1-2\s}\log (km) \ =\ k^{1-\s}(\log k) \sum_{m\le T}m^{1-2\s} +
k^{1-\s} \sum_{m\le T}m^{1-2\s}\log m  
\cr &\le & C_\s k^{1-\s}T^{2(1-\s)}\log (kT) .  
\end{eqnarray*}

Putting both estimates together and operating similarly with the sum corresponding to indices $n,m$ such that $ \ell n>km$ gives
 \begin{eqnarray}\label{rectangle}
\sum_{1\le n,m\le T\atop km\not=\ell n}
\frac{1}{(nm)^\s} \frac{1}{ |\log \frac{km}{\ell n}|}&\le & C_\s T^{2(1-\s)}\big\{1 + k^{1-\s}\log (kT)+ \ell^{1-\s}\log (\ell T)\big\}
 .\end{eqnarray}

 Therefore if $T\ge \frac{k\vee \ell}{(k,\ell})$,
\begin{eqnarray*}
\int_1^T \big|\sum_{1\le n\le t}\frac{1}{n^{\s + it} }\big|^2 \Big(\frac{k}{\ell}\Big)^{it} \dd t &=& T\zeta(2\s) \frac{(k,\ell)^{2\s}}{(k\ell)^\s} + H
\end{eqnarray*}
 with 
$$|H|\le C_\s \Big\{ T\,\frac{(k,\ell)^{2\s}}{(k\ell)^\s}\Big( \frac{T}{\frac{k\vee \ell }{(k,\ell)} }\Big)^{1-2\s}+ T^{2(1-\s)}\big\{1 + k^{1-\s}\log (kT)+ \ell^{1-\s}\log (\ell T)\big\}  \Big\}.$$
And by (\ref{dep}),
\begin{eqnarray*}
\int_1^T |\zeta(\s + it)|^2 \Big(\frac{k}{\ell}\Big)^{it} \dd t&=& T\zeta(2\s) \frac{(k,\ell)^{2\s}}{(k\ell)^\s} + H
\end{eqnarray*}
 with
\begin{eqnarray*}|H|& \le&  C_\s \Big\{T\,\frac{(k,\ell)^{2\s}}{(k\ell)^\s}\Big( \frac{T}{\frac{k\vee \ell }{(k,\ell)} }\Big)^{1-2\s} + T^{2(1-\s)}\big\{1 + k^{1-\s}\log (kT)+ \ell^{1-\s}\log (\ell T)\big\}  \Big\} 
.\end{eqnarray*}
Finally if  $T< \frac{k\vee \ell}{(k,\ell})$, by (\ref{rectangle})
\begin{eqnarray*} \Big|\int_1^T \big|\sum_{1\le n\le t}\frac{1}{n^{\s
+ it} }\big|^2 \Big(\frac{k}{\ell}\Big)^{it} \dd t\Big|& =&\Big|\sum_{1\le n,m\le T\atop km\not=\ell n}\int_{n\vee m}^T \frac{1}{(nm)^\s}\Big(\frac{km}{\ell
n}\Big)^{it}
\dd t\Big|
\cr &\le  &2\sum_{1\le n,m\le T\atop km\not=\ell n}
\frac{1}{(nm)^\s} \frac{1}{ |\log \frac{km}{\ell n}|}
\cr &\le & C_\s T^{2(1-\s)}\big\{1 + k^{1-\s}\log (kT)+ \ell^{1-\s}\log (\ell T)\big\} .
\end{eqnarray*}
 \end{proof}
 \begin{corollary} \label{zeta2}Let   $\s>1/2$. There exists a constant $C_{\s}$ depending on $\s$ only such that for any complex numbers 
$\{c_k,k\in K\}$   and any     real $T\ge 2$, 
we have  
\begin{eqnarray*}
\frac{1}{T}\int_1^T |\zeta(\s + it)|^2 \Big| \sum_{k\in K} c_k k^{it}\Big|^2\dd t&=& \zeta(2\s)
 \sum_{  k, \ell\in K\atop \frac{k\vee \ell}{(k,\ell)}\le T} c_k \overline{c}_\ell \frac{(k,\ell)^{2\s}}{(k\ell)^\s} + H
\end{eqnarray*}
 with
\begin{eqnarray*} |H|
   &\le &C_\s \Big\{  \sum_{k, \ell \in K\atop \frac{k\vee \ell}{(k,\ell)}\le T} |c_k|
|c_\ell|\frac{(k,\ell)^{2\s}}{(k\ell)^\s}\Big( 
\frac{T}{\frac{k\vee \ell }{(k,\ell)} }\Big)^{1-2\s}
 \cr & & \quad +  T^{1-2\s}\Big(\sum_{  \ell \in K } |c_\ell| \Big) \Big(\sum_{  k \in K }|c_k|     k ^{1-\s}\log ( kT)  \Big\}.
\end{eqnarray*}
 \end{corollary} 

\begin{proof}It follows  from Lemma \ref{zeta1} that  
\begin{eqnarray*}
\frac{1}{T}\int_1^T |\zeta(\s + it)|^2 \Big| \sum_{k=1}^N c_k k^{it}\Big|^2\dd t&=&  \zeta(2\s)
 \sum_{ k, \ell \in K\atop \frac{k\vee \ell}{(k,\ell)}\le T} c_k \overline{c}_\ell \frac{(k,\ell)^{2\s}}{(k\ell)^\s} + H
\end{eqnarray*}
 with
\begin{eqnarray*} |H|
  &\le &C_\s \Big\{  \sum_{k, \ell \in K\atop \frac{k\vee \ell}{(k,\ell)}\le T} |c_k|
|c_\ell|\frac{(k,\ell)^{2\s}}{(k\ell)^\s}\Big( 
\frac{T}{\frac{k\vee \ell }{(k,\ell)} }\Big)^{1-2\s}
 \cr & & \quad +  T^{1-2\s}\Big(\sum_{  \ell \in K } |c_\ell| \Big) \Big(\sum_{  k \in K }|c_k|     k ^{1-\s}\log ( kT)  \Big\}.
\end{eqnarray*}

This  allows to conclude easily.

\end{proof}
 \rm  The following corollary is now straightforward.
\begin{corollary}\label{zeta3}Let   $\s>1/2$. There exists a constant $C_{\s}$ depending on $\s$ only such that for any complex numbers $\{c_k, k\in K\}$  
  and any real  $T $  such that $T\ge K^+=\max\{ k, k\in K\}$,    we have  
\begin{eqnarray*}
   (1+ C\frac{K^+ }{T} )\Big(\sum_{k\in K}
|c_k|^2  \Big)  \ \max_{1\le t\le T}  |\zeta(\s + it)|^2    &  \ge & \Big( \zeta(2\s)-\frac{C_\s}{\varpi (K)^{ 2\s-1}}  \Big)
 \sum_{k,\ell =1}^Nc_k \overline{c}_\ell \frac{(k,\ell)^{2\s}}{(k\ell)^\s}  
 \cr  &    &   -  T^{1-2\s}  \Big(\sum_{  k \in K }|c_k|     k ^{1-\s}\log ( kT)  \Big)\sum_{  \ell \in K } |c_\ell|,
\end{eqnarray*} where we have noted
$$\varpi(K)=\max_{ k,\ell \in K}\ \frac{k\vee \ell }{(k,\ell)}.$$
\end{corollary}
\begin{proof}A standard consequence of  Montgomery and Vaughan version of Hilbert's inequality is 
\begin{equation*}   \bigg|\int_0^T \Big|  \sum_{k\in K}   c_kk^{it}\Big|^2\dd t -  T \sum_{k\in K}
|c_k|^2
\bigg|\le \frac{4\pi }{\d}  \sum_{k\in K}|c_k|^2  
\end{equation*}
where $$\d= \min\{ \log \frac{\ell}{k}, \ell >k, \ell, k\in K\}.$$
Plainly $\d\ge 1/CK^+$ where $C$ is an absolute constant. Thus
\begin{equation*}  \frac{1}{T}\int_0^T \Big|  \sum_{k\in K}   c_kk^{it}\Big|^2\dd t \le    \sum_{k\in K}
|c_k|^2 (1+ C\frac{K^+ }{T} ) . 
\end{equation*}  
 This along with  Lemma \ref{zeta2} implies that
\begin{eqnarray*}
  (1+ C\frac{K^+ }{T} )\Big(\sum_{k\in K}
|c_k|^2  \Big)  \ \max_{1\le t\le T}  |\zeta(\s + it)|^2   &  \ge & \Big( \zeta(2\s)-\frac{C_\s}{\varpi (K)^{ 2\s-1}}  \Big)
 \sum_{k,\ell =1}^Nc_k \overline{c}_\ell \frac{(k,\ell)^{2\s}}{(k\ell)^\s}  
  \cr & & \quad -  T^{1-2\s}\Big(\sum_{  \ell \in K } |c_\ell| \Big) \Big(\sum_{  k \in K }|c_k|     k ^{1-\s}\log ( kT)  \Big\}. \end{eqnarray*}
\end{proof}
 
 \begin{proof}[Proof of Theorem \ref{OFC1}]Let $K$ be 	an FC set and choose $c_d= d^{-\e}$, $d\in K$ (that choice  is natural in regard of the form of $b_n$, with $a_d$ in place of $c_d$, see before (\ref{hi})). 
We deduce from (\ref{hi}) that
$$\sum_{n\in K }\frac{1}{n^{2\e}}\s_{-s+\e}( n)^2 \le \zeta(2\s)\sum_{   d,\d\in K}\frac{(d,\d)^{2\s}}{(d\d)^\s}
(d\d)^{-\e} $$
Consequently, if $ K^+ \le T $ and  
$ \varpi (K)\ge c_\s:=(\frac{2C_\s}{\zeta(2\s)  })^{   1/(2\s-1)}  $,
\begin{eqnarray*}
   (1+ C)\Big(\sum_{n\in K } \frac{1}{n^{ 2\e}}  \Big)  \ \max_{1\le t\le T}  |\zeta(\s + it)|^2
  &  \ge & \frac{\zeta(2\s)}{2}
 \sum_{n\in K }\frac{1}{n^{2\e}}\s_{-s+\e}( n)^2  
  \cr & &  - \frac{1}{T^{ 2\s-1}}  \Big(\sum_{  \ell \in K } \frac{1}{\ell^{ \e}} \Big) \Big(\sum_{  k \in K }      k ^{1-\s-\e }\log ( kT)  \Big\}.
\end{eqnarray*} \end{proof}

\begin{proof}[Proof of Theorem \ref{OFC2}]Let $\nu$ be some positive integer. Choose $K$ to be the set of all divisors of $\nu$, which is obviously an FC
set. Then
$$\sum_{n\in K }\frac{1}{n^{2\e}}\s_{-s+\e}( n)^2 =\sum_{n|\nu}\frac{\s_{-s+\e}( n)^2}{n^{2\e}} 
 , \qq \sum_{n\in K } \frac{1}{n^{ 2\e}} =\s_{-2\e} (\nu).$$

Thus for $\nu$ and $T$ such that $\nu \le T$ and $\max_{ [k,\ell ]|\nu}\ \frac{ (k\vee \ell) }{(k, \ell)}\ge c_\s $ 
\begin{eqnarray*}
   (1+ C)  \max_{1\le t\le T}  |\zeta(\s + it)|^2
  &  \ge & \frac{\zeta(2\s)}{2}
\frac{1}{ \s_{-2\e} (\nu)}\sum_{n|\nu}\frac{\s_{-s+\e}( n)^2}{n^{2\e}} 
   -\frac{\s_{-\e} (\nu)\s_{1-\s-\e} (\nu) \log ( \nu T)}{  \s_{-2\e} (\nu)T^{(2\s-1)}}  .
\end{eqnarray*}
 If $\nu$ and $T$ are such that
$$ \frac{\s_{-\e} (\nu)\s_{1-\s-\e} (\nu) \log ( \nu T)}{  \sum_{n|\nu}\frac{\s_{-s+\e}(
n)^2}{n^{2\e}}  }\le \frac{\zeta(2\s)}{4}  T^{(2\s-1)},   $$
we deduce that 
\begin{eqnarray*}
    \max_{1\le t\le T}  |\zeta(\s + it)| 
   &  \ge & c \Big(
\frac{1}{ \s_{-2\e} (\nu)}\sum_{n|\nu}\frac{\s_{-s+\e}( n)^2}{n^{2\e}} \Big)^{1/2}
    ,
\end{eqnarray*}
with $c^2=\frac{\zeta(2\s)}{4 (1+ C) }$ as claimed. 
\end{proof}


\section{\bf Proof of Theorem \ref{c1a}}\label{s4}
  Basically, the principle of the proof consists with showing that the studied case   belongs to
the "domain of attraction" of Carleson's theorem. First, recall for reader's convenience Lemma 8.3.4 from \cite{Wb}. 
\begin{lemma}\label{entro}
Let $\g>1$, $0<\beta\le 1$ and consider a finite
collection of random variables $E=\bigl( X_1,\dots , X_N \bigr)
\subset L^\g(\P)$,  and reals $0\le t_1\le t_2\le \dots \le t_N \le
1$ such that
$$
\|X_j-X_i\|_\g  \le (t_j -t_i)^{\beta }  \qq (\forall 1\le i\le j\le
N).
 $$
Then, there exists a constant $K_{\beta,\g}$
depending on $\beta,\g$ only, such that
 \begin{equation*}
\bigl\|  {\sup_{1\le i,j\le N}   |X_i -X_j |}  \bigr\|_\g
\,\le  \,
\begin{cases}
K_{\beta,\g}   &\text{if $\beta\g>1$,}
\\
K_{\beta,\g} \log  N&\text{if $\beta\g= 1$,}
\\
   K_{\beta,\g} N^{{1\over \g}-\beta}&\text{if $\beta\g<1$.}
\end{cases}
 \end{equation*}
\end{lemma}
This standard Lemma  will be used repeatedly. Let $\{N_j, j\ge 1\}$ be an increasing sequence of integers to be
specified later on.    Let $S_n=
\sum_{k=1}^n c_kf_k$, $n\ge 1$. Put 
\begin{eqnarray*} & & R^J=\sum_{m=1}^J a_me_m,  \qq
\quad r^J=\sum_{m=J+1}^\infty a_me_m 
 .
    \end{eqnarray*}
     We
decompose 
$S_n 
 $ as follows: if $N_j\le n<N_{J+1}$,  then for some $J=J(j)$ depending on  $j$,  the value of which being specified in the course of
the proof, we write that
$$S_n=
\sum_{k=1}^n c_kf_k=\sum_{k=1}^n c_kR^{J }_k+ \sum_{k=1}^n c_kr^J_k  .$$
This way to proceed is not new; we refer for instance to Theorem 2.6 in \cite{BW} where it is used already in the proof. By
Carleson-Hunt's inequality, 
 \begin{eqnarray*}\Big\|\sup_{N_{j  }\le u\le v\le N_{j+1}}\big|\sum_{u\le k\le v}c_kR^J _{ k}
\big|\Big\|_2&\le &\sum_{m=1}^J
|a_{m }|\,\Big\|\sup_{N_{j  }\le u\le v\le
N_{j+1}}\big|\sum_{u\le k\le v}c_ke_{km}\big|\Big\|_2  
 \cr &\le & C\Big(\sum_{m=1}^J |a_{m }|\Big) \Big(\sum_{N_{j  }\le k \le
N_{j+1}} c_k^2\Big)^{1/2} . 
\end{eqnarray*}

 We will show in (\ref{t}), (\ref{iii}) by using Abel summation that the series $\sum_{m\ge 1}^\infty a_{m }^2 d(m)$ is convergent in each
of the considered cases (i)-(iii), and
  we will estimate the
tail
$\sum_{m>L}^\infty a_{m }^2 d(m)$. By Theorem
\ref{c1},  
 \begin{eqnarray*}  \big\|\sum_{u\le k\le v}c_kr^J _{ k}
\big\|_2^2   &\le & \Big( \sum_{m=J+1}^\infty a_{m }^2 d(m)
\Big)    
\sum_{u\le k\le v}  c_{k}^2d(k^2)
 .
    \end{eqnarray*}
By   Lemma \ref{entro} we deduce
\begin{eqnarray*} & &\Big\|\sup_{N_{j  }\le u\le v\le N_{j+1}}\big|\sum_{u\le k\le v}c_kr^J _{ k}
\big|\Big\|_2^2 \cr &\le &   C  (\log N_{j+1})^2\Big( \sum_{m=J+1}^\infty a_{m }^2 d(m)
\Big)  \sum_{N_{j  }\le k\le N_{j+1}}  c_{k}^2d(k^2).
    \end{eqnarray*}
By combining both estimates we arrive to
\begin{eqnarray}\label{comb}& &  \Big\|\sup_{N_{j  }\le u\le v\le N_{j+1}}\big|\sum_{u\le k\le v}c_kf_{ k}
\big|\Big\|_2^2     \le     C\Big(\sum_{m=1}^J |a_{m }|\Big)^2 \Big(\sum_{N_{j  }\le k \le
N_{j+1}} c_k^2\Big) 
\cr  & & \qq\qq\qq\qq + C  (\log N_{j+1})^2\Big( \sum_{m=J+1}^\infty a_{m }^2 d(m)
\Big)  \sum_{N_{j  }\le k\le N_{j+1}}  c_{k}^2d(k^2).
    \end{eqnarray}
Notice that if
 $a_m= o(m^{-\a} )$, $\a>1/2$, 
 we have  
by applying   Abel summation and using the well-known estimate  $ \sum_{m\le \ell}d(m) \le C \ell\log \ell$, 
\begin{eqnarray}\label{t}\sum_{m=L+1}^\infty a_{m }^2 d(m)&\le&\sum_{m=L+1}^\infty \frac{d(m)}{m^{2\a}}\le C\Big\{\sum_{m=L+1}^\infty
\frac{\log m}{m^{2\a}}
  + \sup_{m>L}  m^{1-2\a }\log m\Big\}\cr &\le& C_\a L^{1-2\a}(\log L). 
\end{eqnarray}

Now we give the proof of assertion (i).  Let $1/2<\a<1$. We then deduce from (\ref{comb}), (\ref{t})
\begin{eqnarray}\label{comb1} & &  \Big\|\sup_{N_{j  }\le u\le v\le N_{j+1}}\big|\sum_{u\le k\le v}c_kf_{ k}
\big|\Big\|_2^2  \cr   &\le   &  C\Big(\sum_{N_{j  }\le k \le
N_{j+1}} c_k^2d(k^2)\Big) \Big\{\Big(\sum_{m=1}^J |a_{m }|\Big)^2  
       + C  (\log N_{j+1})^2\Big( \sum_{m=J+1}^\infty a_{m }^2 d(m)
\Big) \Big\} 
\cr &\le   &  C_\a \Big(\sum_{N_{j  }\le k \le
N_{j+1}} c_k^2d(k^2)\Big) \Big\{J^{2(1-\a)}  
   +    (\log N_{j+1})^2 J^{1-2\a }\log J   \Big\}  .
    \end{eqnarray}
We choose $N_j$ so that
$$  (\log N_{j })^{4(1-\a)}  ( \log\log N_{j } )^{2(1-\a)}\sim j^4. $$
 Next choose $J $ so that 
$$\frac{J}{\log J}  
   \sim    (\log N_{j+1})^2   . $$
Then
 $J^{2(1-\a)} \sim   (\log N_{j+1})^{4(1-\a)}(\log\log N_{j+1})^{2(1-\a)} $.  And we obtain
\begin{eqnarray*}  & &  \Big\|\sup_{N_{j  }\le u\le v\le N_{j+1}}\big|\sum_{u\le k\le v}c_kf_{ k}
\big|\Big\|_2^2  
\cr &\le   &  C_\a \Big(\sum_{N_{j  }\le k \le
N_{j+1}} c_k^2d(k^2)\Big)(\log N_{j+1})^{4(1-\a)}(\log\log N_{j+1})}^{2(1-\a)  
 \cr &\le   &  C_\a \Big(\sum_{N_{j  }\le k \le
N_{j+1}} c_k^2{(\log k)^{ 4(1-\a)  }}({\log\log k})^{2(1-\a)} d(k^2)\Big)    
       .
    \end{eqnarray*}
In view of the assumption made, 
we deduce
\begin{eqnarray} \label{acv} \sum_{j} \Big\|\sup_{N_{j  }\le u\le v\le N_{j+1}}\big|\sum_{u\le k\le v}c_kf_{ k}
\big|\Big\|_2^2  
   <\infty     .
    \end{eqnarray}
 By Tchebycheff's inequality and by using Theorem \ref{c1} and (\ref{t}) with $L=2$,
   \begin{eqnarray*}& &\l\Big\{\big|\sum_{N_{j  }< k\le N_{j+1}}  c_kf_k\big|> j^{-2}\Big\}  \le   
 {C_\a}{j^4}  
    \sum_{N_{j  }\le k \le
N_{j+1}} c_k^2   d(k^2) 
\cr &\le & C_\a   \frac{j^4 }{   (\log N_{j })^{4(1-\a)}  ({\log\log N_{j }})^{2(1-\a)} }
     \sum_{N_{j  }\le k \le
N_{j+1}} c_k^2 d(k^2)  (\log N_{j })^{4(1-\a)}  ({\log\log N_{j }})^{2(1-\a)}  \cr &\le &  
  C_\a    
     \sum_{N_{j  }\le k \le
N_{j+1}} c_k^2 d(k^2) (\log k)^{4(1-\a)}  ({\log\log k})^{2(1-\a)}
     . \end{eqnarray*} 
The assumption made implies that 
$$\sum_j\l\Big\{\big|\sum_{N_{j  }< k\le N_{j+1}}  c_kf_k\big|> j^{-2}\Big\}<\infty. $$
 By the Borel-Cantelli lemma,     the series $\sum _j|  \sum_{N_{j  }< u\le 
N_{j+1}}    c_kf _{ k}
 | $    converges almost everywhere. As by (\ref{acv}), 
the oscillation of partial sums around this subsequence is almost surely asymptotically tending to
$0$, this   allows  to conclude.
 
 \vskip 3 pt We continue with giving the  proof of assertion (ii). If $\a=1$, (\ref{comb1}) is slightly modified as follows
\begin{eqnarray}\label{comb11} & &  \Big\|\sup_{N_{j  }\le u\le v\le N_{j+1}}\big|\sum_{u\le k\le v}c_kf_{ k}
\big|\Big\|_2^2  
\cr &\le   &  C  \Big(\sum_{N_{j  }\le k \le
N_{j+1}} c_k^2d(k^2)\Big) \Big\{(\log J)^2  
   +    (\log N_{j+1})^2 \frac{\log J}{J}    \Big\}  .
    \end{eqnarray}
 We choose $N_j$ so that
$$   \log\log N_{j } \sim j^2 .$$
And we choose $J$ so that 
$$ \frac{J} {\log J}  \sim (\log N_{j+1})^2. $$
We deduce
\begin{eqnarray*}   \Big\|\sup_{N_{j }\le u\le v\le N_{j+1}}\big|\sum_{u\le k\le v}c_kf_{ k}
\big|\Big\|_2^2  
 &\le   &  C  \Big(\sum_{N_{j  }\le k \le
N_{j+1}} c_k^2d(k^2)\Big)  (\log\log N_{j+1})^2  
  \cr &\le   &  C  \sum_{N_{j  }\le k \le
N_{j+1}} c_k^2d(k^2)(\log\log k)^2   .
    \end{eqnarray*}
According to the assumption made,
\begin{eqnarray}\label{comb11}\sum_j \Big\|\sup_{N_{j  }\le u\le v\le N_{j+1}}\big|\sum_{u\le k\le v}c_kf_{ k}  \big|\Big\|_2^2<\infty. 
\end{eqnarray}
  By Tchebycheff's inequality and by using Theorem \ref{t1},
   \begin{eqnarray*}\l\Big\{\big|\sum_{N_{j  }< k\le N_{j+1}}  c_kf_k\big|> j^{-2}\Big\} &\le &  
 {C}{j^4}  
    \sum_{N_{j  }\le k \le
N_{J+1}} c_k^2   d(k^2) 
\cr &\le & C\frac{j^4}{(\log\log N_{j+1})^2 }  
     \sum_{N_{j  }\le k \le
N_{j+1}} c_k^2 d(k^2)  (\log\log k)^2    \cr &\le &  
  C     \sum_{N_{j  }\le k \le
N_{j+1}} c_k^2 d(k^2)  (\log\log k)^2
     . \end{eqnarray*} 
 By the Borel-Cantelli lemma,     the series $\sum _j|  \sum_{N_{j  }< u\le  N_{j+1}}    c_kf _{ k}
 | $    converges almost everywhere. This along with (\ref{comb11}) allows  to conclude. 

 \vskip 3 pt Finally we give  the proof of assertion (iii). 
 By using again Abel summation
\begin{eqnarray} \label{iii}\sum_{m>L}  a_{m }^2 d(m)\le C\Big\{\sum_{m>L}  \frac{1}{m(\log m)^{  h}}  +
\sup_{m>L}
\frac{1}{ (\log m)^{ h} } \Big\}\le   \frac{C_h}{ (\log L)^{  h-1} }. 
\end{eqnarray} 
And $\sum_{m=1}^J |a_{m }|\le C_h\sqrt{ J/(\log J)^{1+h}}$. Then by (\ref{comb}), 
\begin{eqnarray}\label{comb12} & &  \Big\|\sup_{N_{j }\le u\le v\le N_{j+1}}\big|\sum_{u\le k\le v}c_kf_{ k}
\big|\Big\|_2^2  \cr   &\le   &  C_h\Big(\sum_{N_{j  }\le k \le
N_{j+1}} c_k^2d(k^2)\Big) \Big\{  \frac{ J}{(\log J)^{1+h}}   
       + C_h   \frac{(\log N_{j+1})^2}{ (\log J)^{  h-1} }
  \Big\} 
  .
    \end{eqnarray}
 We choose $N_j$ so that
  $$       \frac { (\log N_{j+1})^2}{    (\log\log N_{j+1})^{h-1}  } \sim j^4  .$$
  We choose this time $J$ so that $\frac{ J}{(\log J)^{1+h}}\sim \frac{(\log N_{j+1})^2}{ (\log J)^{  h-1} }$, namely 
$$    \frac{ J}{(\log J)^{2}}   
     \sim  (\log N_{j+1})^2   .$$
  Thus $\log J \sim \log\log N_{j+1}$  and 
$$\frac{ J}{(\log J)^{1+h}}\sim \frac{(\log N_{j+1})^2}{ (\log J)^{  h-1} }\sim  \frac{ (\log
N_{j+1})^2}{    (\log\log N_{j+1})^{h-1}  }      .$$

We deduce
\begin{eqnarray*}  & &  \Big\|\sup_{N_{j  }\le u\le v\le N_{j+1}}\big|\sum_{u\le k\le v}c_kf_{ k}
\big|\Big\|_2^2  \cr    &\le   &   \Big(\sum_{N_{j  }\le k \le
N_{j+1}} c_k^2d(k^2)\Big) \Big\{C_h  \frac{ J}{(\log J)^{1+h}}   
       + C'_h   \frac{(\log N_{j })^2}{ (\log J)^{  h-1} }
  \Big\} 
\cr &\le   &  C_h\Big(\sum_{N_{j  }\le k \le
N_{j+1}} c_k^2d(k^2)\Big)   \frac{ (\log N_{j+1})^2}{    (\log\log N_{j+1})^{h-1}  }     
  \cr   &\le   &  C_h \sum_{N_{j  }\le k \le
N_{j+1}} c_k^2d(k^2)    \frac{ (\log k)^2}{    (\log\log k)^{h-1}  }    .
    \end{eqnarray*}
Using the assumption made, it follows that     
$$\sum_{j\ge 1}\ \Big\|\sup_{N_{j  }\le u\le v\le N_{j+1}}\big|\sum_{u\le k\le v}c_kf_{ k}
\big|\Big\|_2^2 <\infty.$$
We conclude by proceeding exactly as before.    Noticing first from estimate (\ref{iii})
 that $\sum_{m\ge 1}  a_{m }^2 d(m)\le      C_h  $, by Tchebycheff's inequality and   Theorem \ref{c1}, we get 
   \begin{eqnarray*}& & 
\l\Big\{\big|\sum_{N_{j  }< k\le N_{j+1}}  c_kf_k\big|>  \frac{1}{j^2} \Big\}\le C
j^4  \sum_{N_{j  }\le k \le N_{j+1}} c_k^2   d(k^2)
\cr &\le &  
 {C_h}j^4  \frac {    (\log\log N_{j+1})^{h-1}  }{ (\log N_{j+1})^2}
    \sum_{N_{j  }\le k \le
N_{j+1}} c_k^2   d(k^2)    \frac{ (\log k)^2}{    (\log\log k)^{h-1}  }
\cr &\le & C_h\sum_{N_{j  }\le k \le
N_{j+1}} c_k^2   d(k^2)   \frac{ (\log k)^2}{    (\log\log k)^{h-1}  } 
     . \end{eqnarray*} 
 By the Borel-Cantelli lemma,     the series $\sum _j|  \sum_{N_{j  }< u\le  N_{j+1}}    c_kf _{ k}
 | $    converges almost everywhere.  We conclude as before.

 
\section{\bf Proof of Theorems
\ref{t1a}, \ref{t1ab}.}\label{s5}
 Let $\{N_j, j\ge 1\}$ be an increasing unbounded sequence of positive reals. 
We write
  $$\sum_{N_j\le k<N_{j+1}}  c_k f_k=\sum_{N_j\le k<N_{j+1}}  c_k R^J_k+ \sum_{N_j\le k<N_{j+1}}  c_k r^J_k$$ where  
$$R^J (x) = \sum_{1\le \ell<J}  \frac{\sin 2\pi \ell x}{\ell }, \qq r^J (x) = \sum_{\ell>J}  \frac{\sin 2\pi \ell x}{j },$$
 and     $J  $ is a real number greater than $1$ and   defined
later on with respect to $j$.
\vskip 2 pt \noi  {\it Proof of Theorem \ref{t1ab}.} Let $b>0$.  We choose 
 $ N_j $ so that
$\log\log N_{j}= j^{\b/b}$ for some $\b>2$.      As
$f\in {\rm BV}(\T)$,
$a_j=
\mathcal O (j^{-1})$, and so
    $$ \sup_{N_{j  }\le u\le v\le N_{j+1}}\big|\sum_{u\le k\le v}c_kR^J _{ k}(x)\big| \le C\sum_{\ell=1}^J\frac{1}{\ell }\sup_{N_{j  }\le
u\le v\le N_{j+1}}\big|\sum_{u\le k\le v}c_k\sin 2\pi k\ell   x\big|. 
$$ 
    By using   Carleson-Hunt's maximal inequality   
\begin{eqnarray}\label{E1}\Big\|\sup_{N_{j  }\le u\le v\le N_{j+1}}\big|\sum_{u\le k\le v}c_kR^J_k \big|\Big\|_2&\le &\sum_{\ell=1}^J
\frac{1}{\ell }\Big\|\sup_{N_{j  }\le u\le v\le
N_{j+1}}\big|\sum_{u\le k\le v}c_k\sin 2\pi k\ell   x\big|\Big\|_2  
 \cr &\le & C(\log J)\, \Big(\sum_{N_{j  }\le u \le
N_{j+1}} c_k^2\Big)^{1/2} . 
\end{eqnarray}

We now combine our Theorem \ref{p1} with the $(\e,1-\e)$ argument introduced in \cite{ABS}. Let $0<\e<1/2$. 
From the bound  
\begin{eqnarray*}   \sum_{i,j>J\atop jk=i\ell} \frac{1}{ij}&\le& C\min \Big(\frac{(k,\ell)}{(k\vee\ell )J}, \frac{(k,\ell)^2}{k\ell}\Big)\le C
\Big(\frac{(k,\ell)}{(k\vee\ell )J} \Big)^\e \Big(  \frac{(k,\ell)^2}{k\ell}\Big)^{1-\e}\cr 
&\le &\frac{C}{ J^\e}
  \Big(  \frac{(k,\ell)^2}{ k\ell  }\Big)^{1-\e/2}=\frac{C}{ J^\e}\langle f^{1-\e/2}_{ k}, f^{1-\e/2}_{\ell}\rangle , 
\end{eqnarray*}
  we get by applying Theorem \ref{p1}-(i),
\begin{eqnarray}\label{rest}\Big\| \sum_{u\le k\le v}c_kr^J_{ k} \Big\|_2^2  &\le &    
 \frac{C}{J^\e}\Big\| \sum_{u\le k\le v}|c_k|f^{1-\e/2}_{ k} \Big\|_2^2  
\cr   &\le&  \frac{C}{J^\e}\Big(\sum_{w \in F([u,v])} \frac{1}{\s_\tau(w) } \Big)\Big(\sum_{{u\le k \le v}}
  c_k^2  \s_{ \tau-2+\e}(k) \Big) . 
\end{eqnarray}
By taking $\tau =1+\e$ and using Corollary \ref{r1} this becomes, 
\begin{eqnarray*}\Big\| \sum_{u\le k\le v}c_kr^J_{ k} \Big\|_2^2 
&\le&  \frac{C}{\e J^\e} \sum_{{u\le k \le v}}
  c_k^2  \s_{ -1+2\e}(k)  . 
\end{eqnarray*}
  By using  Lemma  \ref{entro}, we obtain
\begin{eqnarray}\label{E2}   \Big\|\sup_{N_{j  }\le u\le v\le N_{j+1}} \big| \sum_{u\le k\le v}c_kr^J_{ k} \big|\Big\|_2^2    &\le&
\frac{C(\log N_{j+1})^{2  } }{\e J^\e}    
 \Big(\sum_{N_{j  }\le k \le N_{j+1}}
  c_k^2 \s_{ -1+2\e}(k)  \Big)   . 
\end{eqnarray} 
Combining (\ref{E1}), (\ref{E2}) gives 
\begin{eqnarray}\label{E3}
& & \Big\|\sup_{N_{j  }\le u\le v\le N_{j+1}}\big|\sum_{u\le k\le v}c_kf _{
k}
\big|\Big\|_2^2\cr &\le &  C(\log J)^2\, \Big(\sum_{N_{j  }\le u \le
N_{j+1}} c_k^2\Big) 
  + \frac{C  }{\e J^\e}  (\log N_{j+1})^{2  }  
 \Big(\sum_{N_{j  }\le k \le N_{j+1}}
  c_k^2 \s_{ -1+2\e}(k)  \Big)
\cr &\le &  C\Big(\sum_{N_{j  }\le u \le
N_{j+1}} c_k^2\s_{ -1+2\e}(k)\Big)  \Big((\log J)^2 +\frac{C(\log N_{j+1})^{2  } }{\e J^\e}   \Big)  . 
\end{eqnarray}
  Choose $\e,J$ as follows: 
$$    \log J  =  (\log\log N_{j+1} )^{1+b} , \qq \quad    \e =\frac{  2 }{  (\log\log N_{j+1} )^b} .$$
Then
 $J^\e= \exp \{\e \log J \}=\exp \{2  \log\log N_{j+1} \}= (\log N_{j+1})^{2  }  $. So that
$$\frac{ (\log N_{j+1})^{2  } }{\e J^\e} = \frac {  (\log\log N_{j+1} )^b}{  2 }$$
 We deduce that
\begin{eqnarray*}\Big\|\sup_{N_{j  }\le u\le v\le N_{j+1}} \big| \sum_{u\le k\le v}c_kf_{ k} \big|\Big\|_2^2  
    &\le & C  (\log\log N_{j+1})^{2(1+b)} \sum_{N_{j  }\le k \le N_{j+1}}
  c_k^2   \s_{ -1+2\e}(k)
\cr  &\le&  C_b   \sum_{N_{j  }\le k \le N_{j+1}}
  c_k^2   (\log\log k)^{2(1+b)} \s_{ -1+\frac{4}{(\log\log k)^b}  }(k) . 
\end{eqnarray*}
 By Tchebycheff's inequality,
   \begin{eqnarray*}& &\l\Big\{\sup_{N_{j  }\le u\le v\le N_{j+1}}\big|\sum_{u\le k\le v} c_kf_k \big|> j^{-\b/2}\Big\} 
\cr & \le  &  
  C_b j^{ \b}  \sum_{N_{j  }\le k \le N_{j+1}}
  c_k^2   (\log\log k)^{2(1+b)} \s_{ -1+\frac{4}{(\log\log k)^b}  }(k)  
\cr &\le &   
 \frac{ C_b j^{ \b} (\log\log N_{j+1})^{b}}{(\log\log N_{j+1})^{b}}  \sum_{N_{j  }\le k \le N_{j+1}}
  c_k^2  \s_{ -1+\frac{4}{(\log\log k)^b}  }(k) (\log\log N_{j+1})^{2(1+b)}
     \cr &\le &  
    C_b  \sum_{N_{j  }\le k \le N_{j+1}}
  c_k^2   \s_{ -1+\frac{4}{(\log\log k)^b}  }(k)(\log\log k)^{2 +3b}   
         . \end{eqnarray*} 
By the Borel-Cantelli lemma and the assumption made,     the series 
$$\sum _j  \sup_{N_{j  }\le u\le v\le N_{j+1}}\big|\sum_{u\le k\le v} c_kf_k \big| $$     converges almost everywhere. This allows to
conclude. 
\vskip 10 pt

\vskip 2 pt \noi  {\it Proof of Theorem \ref{t1a}.} The  main change will be in the treatment of the contribution due to the sum $R^J$. Let
$\b>1$. Choose $N_j
=e^{e^{j^B}}$ with
$B=2\b/\d$ and $\d$ is a (small) positive real.   From  estimate
\ref{gron}  follows that,
  $$\s_{ -1+2\e}(k)   \le  \exp\big\{\frac{ \varrho}{  2\e}\frac{(\log k)^{2\e}}{\log\log k}\big\} .$$
 where $\varrho$ is some positive number. 
Thus  (\ref{E2}) with Corollary \ref{r1} gives
\begin{eqnarray*}\Big\| \sum_{u\le k\le v}c_kr^J_{ k} \Big\|_2^2   &\le&    \frac{C}{J^\e} \exp\Big\{\frac{\varrho}{  2\e}\frac{(\log
N_{j+1})^{2\e }}{\log\log N_{j+1}}\Big\}\Big(\sum_{{u\le k \le v}}
  c_k^2   \Big) . 
\end{eqnarray*}
By using again Lemma  \ref{entro} we obtain
\begin{eqnarray*}  & &\Big\|\sup_{N_{j  }\le u\le v\le N_{j+1}} \Big| \sum_{u\le k\le v}c_kr^J_{ k} \Big|\Big\|_2^2  \cr  &\le& \frac{C}{\e J^\e}  
(\log N_{j+1})^{2  }
\exp\Big\{ \frac{\varrho}{  2\e}\frac{(\log N_{j+1})^{2\e }}{\log\log N_{j+1}}\Big\}\Big(\sum_{N_{j  }\le k \le N_{j+1}}
  c_k^2   \Big)   . 
\end{eqnarray*} Choose $\e, J$ so that
$$\e J^\e=    (\log
N_{j+1})^{2  }\exp\Big\{\frac{\varrho}{   \e}\frac{(\log
N_{j+1})^{2\e }}{\log\log N_{j+1}}\Big\} ,\qq \quad  \e =\frac{ \log\log\log N_{j+1} }{2\log\log N_{j+1} } .$$
  We get 
\begin{eqnarray*}\Big\|\sup_{N_{j  }\le u\le v\le N_{j+1}} \Big| \sum_{u\le k\le v}c_kr^J_{ k} \Big|\Big\|_2^2  
  &\le&     \sum_{N_{j  }\le k \le N_{j+1}}
  c_k^2    . 
\end{eqnarray*}
We have 
\begin{eqnarray*}   \log J &=&\frac{1}{\e}\log\frac{  1}{\e}+ (\frac{2}{\e})\log\log N_{j+1}+  \frac{\varrho}{   \e^2}\frac{(\log
N_{j+1})^{2\e }}{\log\log N_{j+1}} 
  \end{eqnarray*}
Further
$$\frac{1}{\e}\log\frac{  1}{\e}=\frac{2\log\log N_{j+1} }{ \log\log\log N_{j+1} }\log \Big(\frac{2\log\log N_{j+1} }{ \log\log\log
N_{j+1} }\Big)\sim 2\log\log N_{j+1},$$ and $$(\log
N_{j+1})^{2\e }=e^{  (\log\log N_{j+1} )(\log\log\log N_{j+1})/ (\log\log N_{j+1} ) }= \log\log N_{j+1} . $$
Thus
  \begin{eqnarray*} \log J &\sim& 2(\log\log N_{j+1})  +  \frac{4 (\log\log N_{j+1})^2}{(\log\log\log N_{j+1})}   +  \frac{\varrho}{  
\e^2}\frac{(\log N_{j+1})^{2\e }}{\log\log N_{j+1}} 
\cr &=&2(\log\log N_{j+1})  +  \frac{4 (\log\log N_{j+1})^2}{(\log\log\log N_{j+1})}   +   4\varrho  \frac{(\log\log N_{j+1})^{2 
}}{(\log\log\log N_{j+1})^2}  
  \cr &\le & C \frac{  (\log\log N_{j+1})^2}{(\log\log\log N_{j+1})}   , \end{eqnarray*}
   
 Now by (\ref{E1}),
\begin{eqnarray*}\Big\|\sup_{N_{j  }\le u\le v\le N_{j+1}}\big|\sum_{u\le k\le v}c_kR^J_{ k} \big|\Big\|_2^2 
 &\le & C  
 (\log J)^2 \Big(\sum_{N_{j  }\le u \le
N_{j+1}} c_k^2\Big) 
\cr  &\le & C  
 \frac{  (\log\log N_{j+1})^4}{(\log\log\log N_{j+1})^2}  \Big(\sum_{N_{j  }\le u \le
N_{j+1}} c_k^2\Big) 
\cr  &\le & C  
    \sum_{N_{j  }\le k \le
N_{j+1}} c_k^2\frac{  (\log\log k)^4}{(\log\log\log k)^2}    . 
\end{eqnarray*}
By combining 
\begin{eqnarray}\label{rmi}\Big\|\sup_{N_{j  }\le u\le v\le N_{j+1}}\big|\sum_{u\le k\le v} c_kf_k \big|\Big\|_2^2 
    &\le & C  
    \sum_{N_{j  }\le u \le
N_{j+1}} c_k^2\frac{  (\log\log k)^4}{(\log\log\log k)^2}       . 
\end{eqnarray}
  By the assumption made, this immediately implies that the series $$\sum_j \sup_{N_{j  }\le u\le v\le N_{j+1}}
 \big|\sum_{u\le k\le v} c_kf_k \big|^2
 $$  converges almost everywhere. And so 
 the oscillation of the partial sum sequence $\{\sum_{k= 1}^{N }c_kf_k, N\ge 1\}$ around the subsequence $\{\sum_{k= 1}^{N_j }c_kf_k, j\ge
 1\}$ tends to zero almost everywhere.

 By Tchebycheff's inequality,
   \begin{eqnarray*} \l\Big\{\sup_{N_{j  }\le u\le v\le N_{j+1}}\big|\sum_{u\le k\le v} c_kr^J_{ k} \big|> j^{-\b}\Big\} 
  &\le &  
  C j^{2\b}  
     \sum_{N_{j  }\le k \le
N_{j+1}} c_k^2   
  \cr&\le &  C   \frac{j^{2\b} }{  (\log\log N_{j+1}) }       \sum_{N_{j  }\le u
\le N_{j+1}} c_k^2\log\log k
 \cr&\le &  {C}
    \sum_{N_{j  }\le u \le
N_{j+1}} c_k^2     (\log\log k) 
     . \end{eqnarray*} 
By the Borel-Cantelli lemma,     the series 
$$\sum _j|  \sum_{N_{j  }< u\le  N_{j+1}}    c_kr^J_{ k}
 | $$
     converges almost everywhere. 
 We shall prove that the series $$\sum _j|  \sum_{N_{j  }< u\le  N_{j+1}}    c_kR^J_{ k}
 | $$ also converges almost everywhere. This part is more tricky.
We begin with a remark concerning the sum  related to the  component $R^J$. Recall   that 
$$ \langle R^J_k, R^J_\ell \rangle =
\sum_{i,h=1\atop  hk= i\ell}^J\frac{1}{ih}= \Big(
\sum_{1\le u\le J(\frac{ (k,\ell)}{k}\wedge
\frac{ (k,\ell))}{\ell}}^J\frac{1}{u^2}\Big)\frac{ (k,\ell)^2 }{k \ell}.$$ 
   The existence of a solution in $h$ and $i$  (automatically of the form $i=u\frac{k}{(k,\ell)},
h=u\frac{\ell  }{(k,\ell)}
$, $u\ge 1$)  imposes constraints on the integers 
$k,\ell$. These  must satisfy the following conditions 
$$\ell\le u\ell= (k,\ell)h\le J(k,\ell) ,\qq k\le uk=(k,\ell)i\le J(k,\ell),$$
that is    $ (k,\ell)\ge  \frac{(k\vee\ell)}{J} $. 
 Consequently, as obviously $ (k,\ell)\le (k\wedge \ell)$, it is necessary  to have
$$  \frac{1}{J}(k\vee\ell)\le (k\wedge \ell) .  $$
In this case  $0\le \langle R^J_k, R^J_\ell \rangle \le 
\zeta(2)\frac{ (k,\ell)^2 }{k \ell}$. Observe before continuing that in our situation $J\ll N_{j+1}$ while $N_j<k,\ell \le N_{j+1}$. 

Let $ h$ and $H$ be such that $J^{h} < N_j\le
J^{h+1}\le\ldots \le  J^{h+H-1} \le N_{j+1}< J^{h+H}$. It follows from the remark previously made and estimate (\ref{61}) that,  
\begin{eqnarray*}\big\| \sum_{N_j<k\le N_{j+1}}c_kR^J_k\big\|_2^2
&\le& \zeta(2 )\sum_{N_j<k,\ell \le N_{j+1}\atop   (k\vee\ell)\le {J} (k\wedge \ell)}|c_k||c_\ell | \frac{ (k,\ell)^2 }{k \ell}
\cr &= & \zeta(2 )\sum_{\m=h}^H\sum_{k\in  ]J^{\m} ,J^{\m+1}] \atop  \ell: (k\vee\ell)\le {J} (k\wedge \ell)}|c_k||c_\ell | \frac{ (k,\ell)^2 }{k \ell}
\cr &\le  & \zeta(2 )\sum_{\m=h}^H\ \sum_{ J^{\m}<k\le  J^{\m+1} \atop  \frac{1}{J}.J^{\m } \le \ell\le   J.J^{\m+1}   }|c_k||c_\ell | \frac{ (k,\ell)^2 }{k \ell}
\cr &\le  & \zeta(2 )\sum_{\m=h}^H\ \sum_{  J^{\m-1} \le k, \ell\le   J^{\m+2}   }|c_k||c_\ell | \frac{ (k,\ell)^2 }{k \ell}
\cr &\le  & \big(4\zeta(2 )  \log J \big) \sum_{\m=h}^H\sum_{   J^{\m-1} \le k \le   J^{\m+2}     } c_k^2 \s_{-1}(k)
\cr &\le   & C  \sum_{\m=h}^H\sum_{   J^{\m-1} \le k \le   J^{\m+2}     }    c_k^2 \frac{  (\log\log
k)^2}{ \log\log\log k }\s_{-1}(k)
\end{eqnarray*} 
 since  $N_j<k  \le N_{j+1}$ and
 \begin{eqnarray*} \log J &\sim&  4(1+o(1)) \frac{  (\log\log N_{j+1})^2}{(\log\log\log N_{j+1})} \sim 4(1+o(1)) \frac{  (\log\log
k)^2}{ \log\log\log k }
  \end{eqnarray*}
Therefore\begin{eqnarray}\label{Rest}\big\| \sum_{N_j<k\le N_{j+1}}c_kR^J_k\big\|_2^2
  &\le&   C  \sum_{ J^{-1}N_j <k\le  N_{j+1}J^2  }    c_k^2 \frac{  (\log\log
k)^2}{ \log\log\log k }\s_{-1}(k).
\end{eqnarray} 
 By Tchebycheff's inequality,
   \begin{eqnarray*}& &\l\Big\{ \big|\sum_{N_{j  }< k\le N_{j+1}} c_kR^J_k \big|> j^{-\b}\Big\} 
\cr &\le &  
 {C}j^{2\b}     \sum_{ J^{-1}N_j <k\le  N_{j+1}J^2  }    c_k^2 \frac{  (\log\log
k)^2}{ \log\log\log k }\s_{-1}(k)
    \cr &\le &  
 {C}     {  j^{2\b-B\d}(\log\log N_{j+1}})^\d   \sum_{ J^{-1}N_j <k\le  N_{j+1}J^2  }    c_k^2 \frac{  (\log\log
k)^2}{ \log\log\log k }\s_{-1}(k)  
 \cr &\le &  
 {C}     \sum_{ J^{-1}N_j <k\le  N_{j+1}J^2  }    c_k^2 \frac{  (\log\log
k)^{2+\d}}{ \log\log\log k }\s_{-1}(k)   \end{eqnarray*} 
Recall that $J=J(j)$ is associated with the interval $]N_j ,  N_{j+1}]$. Now observe that 
$$   \frac{N_{j+2}}{J(j+2)} > N_{j+1}J(j)^2.$$
Indeed, 
$$  J(j)^2J(j+2)\sim e^{\frac{8+o(1)}{B}\frac{j^B}{\log j}+\frac{8+o(1)}{B}\frac{(j+2)^B}{\log (j+2)}} .$$
Thus $J(j)^2J(j+2)\le e^{C_B j^B}$,  whereas for $j$ large
$$     \frac{N_{j+2}}{N_{j+1} }=e^{e^{(j+2)^B}-e^{(j+1)^B}}\ge  e^{e^{(j+2)^B}/2}
\gg J(j)^2J(j+2).$$
This means that the intervals $]J(j)^{-1}N_j,  N_{j+1}J(j)^2]$ $j= 2,4,6\ldots$ are disjoint. The same holds for the sequence of  intervals with odd indices.
Consequently, 
$$\sum _j\l\Big\{\big|\sum_{N_{j  }< k\le N_{j+1}} c_kR^J_k \big|> j^{-\b}\Big\}\le C \sum_{k  }    c_k^2 \frac{  (\log\log
k)^{2+\d}}{ \log\log\log k }\s_{-1}(k)<\infty, $$
by assumption. Hence by the Borel-Cantelli lemma, the series  
$$\sum_{j}\big|\sum_{N_{j  }< k\le N_{j+1}} c_kR^J_k \big|$$
converges almost everywhere. This allows to conclude. 
\begin{remark} It is interesting to notice that from estimates (\ref{E1}) and (\ref{Rest}) and Gronwall's estimate (\ref{gron1}),
\begin{eqnarray*}\Big\|\sup_{N_{j  }\le u\le v\le N_{j+1}}\big|\sum_{u\le k\le v}c_kR^J_k \big|\Big\|_2^2&\le & C  
    \sum_{N_{j  }\le k \le
N_{j+1}} c_k^2\frac{  (\log\log k)^4}{(\log\log\log k)^2} . 
\end{eqnarray*}
while \begin{eqnarray*}\big\| \sum_{N_j<k\le N_{j+1}}c_kR^J_k\big\|_2^2
  &\le&   C  \sum_{ J(j)^{-1}N_j <k\le  N_{j+1}J(j)^2  }    c_k^2 \frac{  (\log\log
k)^3}{ \log\log\log k }.
\end{eqnarray*}\end{remark}
\begin{remark} The following set  $B = \big\{k,\ell ;  N<k,\ell \le M:J(k,\ell)\ge    (k\vee\ell)\big\}$ appeared in the last step of the proof.
Concerning the size of such sets, one can show the general bound
 $$\#(B )   \le   CJM (\log JM)^3,$$
where $C$ is absolute. \end{remark}

\section{\bf Proof of Theorem \ref{sw}.}\label{s6} 
Changing
$f$ for
$f/ c$ if necessary, we may   assume for our purpose that
$c=1$ in (\ref{qos}).  Let $G_n=(\g_{k,\ell}  )$, where $\g_{k,\ell}=\int _Ef_k  f_\ell\, \dd x $ denotes the Gram matrix of the
system $f_1, \ldots, f_n$.  As 
$H_n=I-G_n$  is nonnegative definite, there exist in $\R^n$   vectors $u_1,\ldots, u_n$ with Gram matrix $H_n$, for instance the rows of
$H_n^{1/2}$. Given any bounded interval
$Y$, it follows that there exist in $L^2(Y)$ (in fact in any separable Hilbert space),  vectors $v_1,\ldots, v_n$ with Gram matrix
$I-G_n$. By induction (using isometry), it is plain that if   $v_1,\ldots, v_n$ are already chosen with Gram matrix $H_n$, a vector
$v_{n+1}$ can be added so that the new system 
 $v_1, \ldots, v_{n+1}$  will have    Gram matrix $H_{n+1}$.  
    Consequently  there exist $(g_k)$  supported on $Y$ such that $(f_k+ g_k)  $ is an orthonormal system
on
$\T\times Y$. Thus for any
$(c_k)$ universal, the series
$\sum_k c_k (f_k+ g_k)  $ converges a.e. on $\T\times Y$, and thereby converges a.e. on $\T$.  
Since
$g_k\equiv 0$ on $\T$, it follows that
$\sum_k c_k  f_k   $ converges a.e.
 \begin{remark} The construction of  $(g_k)$ is exactly as in the proof of Schur's Lemma (\cite{O}, p.\,56).
\end{remark}
\section{ A strenghtened form of Theorem \ref{p1}}\label{s7}
Our goal in this section is to show that the form of the upper bound provided in Theorem \ref{p1} in fact strongly depends on how the considered GCD quadratic form can be bounded from below. This is a quite  intriguing property, expressed in Corollary \ref{c1a1}, and  which we will study more thoroughly elsewhere. We first establish the following  stronger estimate. 
\begin{theorem} \label{t1a1}Let  $s>1/2$ and  $0\le \tau \le 2s$.  For any finite coefficient sequence $\{c_k, k\in K\}$,
\begin{eqnarray*} 
  & & \Big|\sum_{k,\ell\in K}c_k   c_\ell
 \frac{(k,\ell)^{2s}}{k^s\ell^s} 
-\sum_{ k \in K}  c_k^2 \Big| \cr   &\le&   \Big(\sum_{k  \in K}\frac{|c_k | }{k^s }  \Big)^2 +2 \big(\zeta(2s)\sum_{ k \in K^*}  c_k^2
 \big)^{1/2}\Big( \sum_{k,\ell\in K}|c_k  || c_\ell|
 \frac{(k,\ell)^{2s}}{k^s\ell^s} 
\Big)^{1/2}+ \cr & & \Big(\sum_{u \in F'(K)} \frac{1}{\s_\tau(u) } \Big)\Big(\sum_{\nu \in K}
  c_\nu^2  \s_{ \tau-2s}(\nu) \Big), 
 \end{eqnarray*}
 where $F'(K)=F(K)\backslash \{1\}$ and $K^*=\{\ell \in K; \exists k\in K, k< \ell : k|\ell \}$. \end{theorem}
\begin{remark}In general, $K^*  $ is a significantly smaller set than $K$.  If $K=[a,b]$ with $b> a> 1$, then 
  $K^*\not\subseteq K$, then $a+1\notin K^*$ simply because $a\!\not| a+1$.    The extremal case corresponds to Rudin sets,  namely
sets    of integers  
  none of
which divides the least common multiple   of the
others, in which case $K^* =\emptyset $.
\end{remark}
\begin{corollary}\label{c1a1} Let  $s>1/2$ and  $0\le \tau \le 2s$. Let $\rho =\sum_{ \l \in K}\frac{1}{\l^{2s}}$. Assume $c_k\ge 0$, $k\in
K\}$ and $\rho \le 1/16$.  We have the following alternative. Either
\begin{eqnarray*} 
 \sum_{k,\ell\in K}c_k   c_\ell
 \frac{(k,\ell)^{2s}}{k^s\ell^s}   &\le&\frac{1}{\sqrt \rho}\Big(\sum_{k  \in K}\frac{c_k  }{k^s } +\sum_{u \in F'(K)} \frac{1}{\s_\tau(u) } +\sum_{\nu \in K}
  c_\nu^2  \s_{ \tau-2s}(\nu) \Big)^2, 
 \end{eqnarray*}
or
\begin{eqnarray*} 
 \sum_{k,\ell\in K}c_k   c_\ell
 \frac{(k,\ell)^{2s}}{k^s\ell^s}   &\le& \frac{1}{1-3\sqrt \rho}\sum_{ \ell \in K}  c_\ell^2. 
 \end{eqnarray*} 
\end{corollary}
   \begin{proof}[Proof of Theorem \ref{t1a1}] From (\ref{HS1a})  and M\"obius inversion formula (\ref{m}),
   \begin{eqnarray*}  
  \sum_{k,\ell\in K}c_k   c_\ell
 \frac{(k,\ell)^{2s}}{k^s\ell^s} &=& \sum_{k,\ell \in K}  \frac{c_k   c_\ell }{k^s\ell^s}\Big\{\sum_{d\in F(K)} 
J_{2s} (d)  {\bf 1}_{d|k} {\bf 1}_{d|\ell}\Big\}  
\cr &=& \sum_{ \ell \in K}  \frac{  c_\ell^2 }{ \ell^{2s}}\Big\{\sum_{d\in F(K)} 
J_{2s} (d)    {\bf 1}_{d|\ell}\Big\}  +\sum_{k,\ell \in K\atop k\not= \ell}  \frac{c_k   c_\ell }{k^s\ell^s}\Big\{\sum_{d\in F(K)} 
J_{2s} (d)  {\bf 1}_{d|k} {\bf 1}_{d|\ell}\Big\}
\cr &=& \sum_{ \ell \in K}    c_\ell^2   +\sum_{k,\ell \in K\atop k\not= \ell}  \frac{c_k   c_\ell }{k^s\ell^s}\Big\{\sum_{d\in F(K)} 
J_{2s} (d)  {\bf 1}_{d|k} {\bf 1}_{d|\ell}\Big\}    .
 \end{eqnarray*}
  We begin with isolating particular values of $d$ in the sum in brackets. 
Observe that the case 
$d=1$  contributes at most for 
\begin{eqnarray}  \label{c2}
 \Big(\sum_{k  \in K}\frac{|c_k | }{k^s }  \Big)^2 &\le & \zeta(2s) \sum_{k  \in K} |c_k|^2   , 
 \end{eqnarray}
 by using Cauchy-Schwarz's inequality. Now if $d=k$ or $d=\ell$, then $J_{2s} (d)  {\bf 1}_{d|k} {\bf 1}_{d|\ell}$ is either
equal to
$$  J_{2s} (k)    {\bf 1}_{k|\ell}\qq {\rm or} \qq  J_{2s} (\ell)    {\bf 1}_{\ell|k } $$
These cases contribute for
\begin{eqnarray*}  2\sum_{ \ell \in K }\sum_{k  \in K\atop k<\ell, k|\ell}  \frac{c_k   c_\ell }{k^s\ell^s} J_{2s} (k)
  .
     \end{eqnarray*}
We have\begin{eqnarray} 
\label{c3} \Big|\sum_{
\ell
\in K }\sum_{k 
\in K\atop k|\ell} 
\frac{c_k   c_\ell }{k^s\ell^s} J_{2s} (k)\Big|&  =  &  \Big|\sum_{ \ell \in K^* }\sum_{k  \in K\atop k|\ell}  c_k   c_\ell\Big(\frac{ k  }{ \ell
}\Big)^s\frac{J_{2s} (k)   }{k^{2s}} \Big|
\cr &\le &  \sum_{ \ell \in K^* }|c_\ell|\sum_{k  \in K\atop k<\ell, k|\ell} |c_k |   \Big(\frac{ k  }{ \ell }\Big)^s\ =\  \sum_{ \ell \in K^*
}|c_\ell| b_\ell ,     
\end{eqnarray}
with 
$$ b_\ell=\sum_{k  \in K\atop k|\ell} |c_k |   \Big(\frac{ k  }{ \ell }\Big)^s.$$  
Operating as in \cite{H}  (proof of Proposition 3.1), 
 \begin{eqnarray*} \sum_{ \ell \in K^*}  b_\ell^2&=& \sum_{ \ell \in K^*}  \sum_{u,v  \in K\atop [u,v]|\ell} | c_u || c_v | 
\frac{ u^s  v^s  }{ \ell^{2s} }   = \sum_{u,v  \in K } |c_u || c_v |\frac{ u^s  v^s  }{ [u,v]^{2s} }\sum_{ \l \in K^*_{[u,v]} } 
 \frac{ 1 }{ \l^{2s} }   .    \end{eqnarray*}
  By the Cauchy-Schwarz inequality,
\begin{eqnarray}  \label{c4}\Big|\sum_{ \ell \in K^*
}c_\ell b_\ell\Big|&\le &\big(\sum_{ \ell \in K^*}  c_\ell^2
 \big)^{1/2}\big(\sum_{ \ell \in K^*}  b_\ell^2
 \big)^{1/2}\cr
& \le &    \Big(\sum_{ \ell \in K^*}  c_\ell^2
 \Big)^{1/2}\Big(\sum_{u,v  \in K }| c_u||  c_v | \frac{ (u,v)^{2s} }{ u^s  v^s  }\sum_{ \l \in K^*_{[u,v]} } 
 \frac{ 1 }{ \l^{2s} } \Big)^{1/2}. \end{eqnarray}
 \vskip 2 pt We next concentrate on the sum 
\begin{eqnarray*}  
 \Sigma:= \sum_{k,\ell \in K}  \frac{c_k   c_\ell }{k^s\ell^s}\Big\{\sum_{d\in F(K)\atop d\notin\{1,k,\ell\}} 
J_{2s} (d)  {\bf 1}_{d|k} {\bf 1}_{d|\ell}\Big\} =\sum_{d\in F'(K) } 
J_{2s} (d) \sum_{k,\ell \in K\atop
k\not=d, \ell\not = d}  \frac{c_k   c_\ell }{k^s\ell^s} {\bf 1}_{d|k} {\bf 1}_{d|\ell}
  . 
 \end{eqnarray*}
  Writing $k=ud$, $\ell=vd$ and noticing that
$u,v\in F'(K)$, we have 
\begin{eqnarray*} 
\Big|\sum_{d\in F'(K) } 
J_{2s} (d) \sum_{k,\ell \in K\atop
k\not=d, \ell\not = d}  \frac{c_k   c_\ell }{k^s\ell^s} {\bf 1}_{d|k} {\bf 1}_{d|\ell}\Big|&\le &\sum_{d\in F'(K) } 
\frac{J_{2s} (d)}{d^{2s}}  \sum_{u,v\in F'(K)} \frac{|c_{ud}||c_{vd} |}{u^sv^s}  
\cr  &=& \sum_{u,v\in F'(K)} \frac{1}{u^sv^s} \Big(\sum_{d\in F'(K)}
\frac{J_{2s} (d)}{d^{2s}}|c_{ud}||c_{vd} |  \Big) . 
 \end{eqnarray*}
By the Cauchy-Schwarz inequality,
$$ \sum_{d\in F'(K)} \frac{J_{2s}
(d)}{d^{2s}}|c_{ud}||c_{vd} |   \le \Big(\sum_{d\in F'(K)} \frac{J_{2s}
(d)}{d^{2s}}c_{ud}^2  \Big)^{1/2}\Big(\sum_{d\in F'(K)} \frac{J_{2s}
(d)}{d^{2s}} c_{vd}^2   \Big)^{1/2}.$$
Hence,
\begin{eqnarray*} 
\Sigma  &\le& \Big[\sum_{u \in F'(K)} \frac{1}{u^s } \Big(\sum_{d\in F'(K)} \frac{J_{2s}
(d)}{d^{2s}}c_{ud}^2  \Big)^{1/2}\Big]^2 . 
 \end{eqnarray*}
Now we continue as in the proof of Theorem \ref{p1}. Let   $\psi $ be a positive arithmetic function. Writing
$$ \frac{1}{u^s }= \frac{1}{ u^{s/2}\psi(u) ^{1/2} }.\,\frac{\psi(u) ^{1/2}}{u^{s/2} }$$
and applying  Cauchy-Schwarz's inequality again gives,
\begin{eqnarray*} 
\Sigma  &\le& \Big(\sum_{u \in F'(K)} \frac{1}{u^s\psi(u) } \Big)\Big(\sum_{u \in F'(K)}
\frac{\psi(u)}{u^s  } \sum_{d\in F'(K)}
\frac{J_{2s} (d)}{d^{2s}}c_{ud}^2  \Big)
\cr &\le& \Big(\sum_{u \in F'(K)} \frac{1}{u^s\psi(u) } \Big)\Big(\sum_{u \in F(K)}
\frac{\psi(u)}{u^s  } \sum_{d\in F(K)}
\frac{J_{2s} (d)}{d^{2s}}c_{ud}^2  \Big)  . 
 \end{eqnarray*}
We note that the first sum is indexed on $F'(K)$ this time. The second sum was already estimated.  We recall that
\begin{eqnarray*} 
\sum_{u \in F(K)}
\frac{\psi(u)}{u^s  } \sum_{d\in F(K)}
\frac{J_{2s} (d)}{d^{2s}}c_{ud}^2    &\le&  \sum_{\nu \in K} \frac{ c_\nu^2}{  \nu^{2s}   }   \sum_{u \in
F(K)\atop u|\nu }
  J_{2s}\big( \frac{\nu}{u   }\big) u^{ s} \psi(u)    , 
 \end{eqnarray*}
and for   $\psi(u) = u^{-s} \s_\tau(u)$ with $\tau\le 2s$ 
 \begin{eqnarray*} 
     \sum_{u \in
F(K)\atop u|\nu }
  J_{2s}\big( \frac{\nu}{u   }\big) u^{ s} \psi(u) &=& \nu^\tau \s_{2s-\tau}(\nu). 
 \end{eqnarray*}  
 So that
\begin{eqnarray*} 
\Sigma  &\le& \Big(\sum_{u \in F'(K)} \frac{1}{\s_\tau(u) } \Big)\Big(\sum_{\nu \in K}
\frac{ c_\nu^2}{  \nu^{2s}   }\nu^\tau \s_{2s-\tau}(\nu) \Big) . 
 \end{eqnarray*}
Finally,\begin{eqnarray} \label{c5}
\Sigma 
&\le& \Big(\sum_{u \in F'(K)} \frac{1}{\s_\tau(u) } \Big)\Big(\sum_{\nu \in K}
  c_\nu^2  \s_{ \tau-2s}(\nu) \Big), 
 \end{eqnarray}
By combining (\ref{c2}), (\ref{c4}), (\ref{c5}) we get
\begin{eqnarray}\label{ineq}  
& &  \Big|\sum_{k,\ell\in K}c_k   c_\ell
 \frac{(k,\ell)^{2s}}{k^s\ell^s} -\sum_{ \ell \in K}    c_\ell^2  \Big|  
    \le   \Big(\sum_{k  \in K}\frac{c_k  }{k^s }  \Big)^2 \cr & & +2 \Big(\sum_{ \ell \in K^*}  c_\ell^2
 \Big)^{1/2}\Big(\sum_{u,v  \in K } c_u  c_v \frac{ (u,v)^{2s} }{ u^s  v^s  } \sum_{ \l \in K^*_{[u,v]} } 
 \frac{ 1 }{ \l^{2s} } \Big)^{1/2}
  \cr & & + \Big(\sum_{u \in F'(K)} \frac{1}{\s_\tau(u) } \Big)\Big(\sum_{\nu \in K}
  c_\nu^2  \s_{ \tau-2s}(\nu) \Big). 
 \end{eqnarray}
 \end{proof}
 \begin{proof}[Proof of Corollary \ref{c1a1}]    Put
 \begin{eqnarray*}\qq \qq \begin{cases}A= \big(\sum_{u \in F'(K)} \frac{1}{\s_\tau(u) } \big)\big(\sum_{\nu \in K}
  c_\nu^2  \s_{ \tau-2s}(\nu) \big)
\cr 
L=\sum_{k,\ell\in K}c_k   c_\ell
 \frac{(k,\ell)^{2s}}{k^s\ell^s}\cr 
 \e =\sum_{ \ell \in K}  c_\ell^2. \end{cases}\end{eqnarray*}
By (\ref{ineq}),
\begin{eqnarray*}
  |L -\e  |  
   & \le &  \Big(\sum_{k  \in K}\frac{c_k  }{k^s }  \Big)^2 
   \cr &\le & +2 \Big(\sum_{ \ell \in K^*}  c_\ell^2
 \Big)^{1/2}\Big(\sum_{u,v  \in K } c_u  c_v \frac{ u^s  v^s  }{ [u,v]^{2s} }\sum_{ \l \in K^*_{[u,v]} } 
 \frac{ 1 }{ \l^{2s} } \Big)^{1/2} + A
  \cr &\le &\Big(\sum_{k  \in K}\frac{c_k  }{k^s }  \Big)^2 +A+\big(L\e\rho)^{1/2}. 
 \end{eqnarray*}
 Let $0<h<1$. Either 
 $(1-h)L< \e$. Thus
$$L\le \frac{1}{1-h}\sum_{ \ell \in K}  c_\ell^2.$$
Or $(1-h)L\ge \e$, and so 
 $$L\le \big(\sum_{k  \in K}\frac{c_k  }{k^s }  \big)^2+(1-h)L+2\big((1-h)\rho)^{1/2}L+A. $$Thus
$$L\big(h-2\big((1-h)\rho)^{1/2}\big)\le  \big(\sum_{k  \in K}\frac{c_k  }{k^s }  \big)^2+A,
$$
Let $h= 3\sqrt \rho$. Then $h-2\big((1-h)\rho)^{1/2}\ge \sqrt \rho $ and  $(1-h)\ge 1/4\ge \sqrt \rho $.  Now
\begin{eqnarray*}
 \big(\sum_{k  \in K}\frac{c_k  }{k^s }  \big)^2+A&= &\big(\sum_{k  \in K}\frac{c_k  }{k^s }  \big)^2+\Big(\sum_{u \in F'(K)} \frac{1}{\s_\tau(u) } \Big)\Big(\sum_{\nu \in K}
  c_\nu^2  \s_{ \tau-2s}(\nu) \Big)\cr
  &\le &  \big(\sum_{k  \in K}\frac{c_k  }{k^s }  \big)^2+\frac{1}{2}\Big(\sum_{u \in F'(K)} \frac{1}{\s_\tau(u) } +\sum_{\nu \in K}
  c_\nu^2  \s_{ \tau-2s}(\nu) \Big)^2
\cr
  &\le &\Big(\sum_{k  \in K}\frac{c_k  }{k^s } +\sum_{u \in F'(K)} \frac{1}{\s_\tau(u) } +\sum_{\nu \in K}
  c_\nu^2  \s_{ \tau-2s}(\nu) \Big)^2.
\end{eqnarray*} 
Hence
$$L\le \frac{1}{\sqrt \rho}\Big(\sum_{k  \in K}\frac{c_k  }{k^s } +\sum_{u \in F'(K)} \frac{1}{\s_\tau(u) } +\sum_{\nu \in K}
  c_\nu^2  \s_{ \tau-2s}(\nu) \Big)^2.$$
\end{proof}
 
\vskip 5 pt \noi {\bf Acknowlegments:} I   thank the Center of Advanced Study from  Oslo and Kristian Seip for  inviting me in April 2013, where this project was initiated. I wish  to thank
  Kristian Seip for discussions around the new approach using analysis on the polydisc implemented in \cite{{ABS}} (and quite recently in  \cite{{BoS}}).
This  motivated me in the project to propose an alternative approach using convolution calculus, and developing somehow the works \cite{{We}, {BW1}}. 
 I  am  pleased to thank  Christophe Cuny for reading a preliminary version of the paper
and for useful comments. I also thank Julien Br\'emont for useful comments.
\vskip 5 pt
{\it Final Note.}   In a very recent work, Lewko and Radziwill (arXiv:1408.2334v1) proposed a new, simple (with no combinatorial argument) and very
informative approach to G\'al's theorem and the recent extensions obtained in
\cite{ABS}. They further applied their results to the convergence almost everywhere of series of dilates of functions in $ {\rm BV}(\T)$, and were able to
reduce the condition $\g>4$ in (\ref{abs4}) to $\g>2$. This naturally include our Theorem \ref{t1a}, but not our Theorem \ref{c1a}  with arithmetical
multipliers. Further, the new argument we introduced in the  proof of Theorem \ref{t1a} suggests a possibility to improve Lewko and Radziwill's
convergence condition by  requiring only that $\sum_{k}  c_k^2\frac{(\log\log k)^2}{(\log\log \log k)^2}  <\infty$. This will be investigated
elsewhere.

We also mention Aistleitner's recent result  (arXiv:1409.6035v1)
deriving from Hilberdink's approach,  a lower bound for the maximum of $| \zeta(\s + it)|$, $0\le t\le T$, $T$  large,   of the same kind as Montgomery well-known result, with slightly better constant. 



\begin{thebibliography}{99}
\bibitem{ABS}  Aistleitner  C., Berkes  I., Seip. K. (2013)  {\sl GCD sums from
Poisson integrals and systems of dilated functions}. \emph{J. Eur. Math. Soc.}, to appear.  
\bibitem{ABSW}  Aistleitner  C., Berkes  I., Seip. K., Weber M. (2014)  {\sl  Convergence of  series of dilated functions  and spectral norms of GCD matrices}, submitted.
 \bibitem{BS}  Balazard M., Saias S. (2004) {\sl Notes sur la fonction $\zeta$ de Riemann, 4}, Advances in Math. {\bf 188}, 
69--86.  
    \bibitem{Br}  Br\'emont J. (2011) {\sl Davenport series and almost sure convergence},  Quart. J. Math. {\bf 62},
825--843.
\bibitem{BW2}  Berkes I., Weber M. (2013) {\sl On series $\sum c_k f(kx)$ and
Khinchin's conjecture}, Isra\"el J. Math, to
appear.
\bibitem{BW}  Berkes I., Weber M. (2012) {\sl On series of dilated functions}, Quartely J. Math, to
appear.
\bibitem{BW1}  Berkes I., Weber M. (2009) {\sl On the convergence of $\sum c_k f(n_k x)$},    Memoirs  of the  A.M.S. {\bf 201} no. {\bf
943},
   vi+72p.
\bibitem{Bi} Billingsley, P.   (1999) {\it Convergence of probability measures},
 second edition, Wiley Ser. Probab. Statist. Prob. Stat., John Wiley \& Sons, Inc., New York. 
\bibitem{BoS}  Bondarenko A., Seip K. (2013) {\sl GCD sums and complete sets of square-free numbers}, arXiv:14.02.0249.  
\bibitem{C} Carlitz J. (1960) {\sl   Some matrices related to the greatest integer function},    J. Elisha Mitchell Sci. Soc. 
  {\bf 76},    5--7.
 \bibitem{Ce1} Ces\`aro E. (1885) {\sl   Consid\'erations nouvelles sur le d\'eterminant de Smith et Mansion},    Annales scientifiques
de l'\'Ecole normale sup\'erieure, Paris   {\bf 2}, 425--435. 
 \bibitem{Ce2} Ces\`aro E. (1886) {\sl   Formes alg\'ebriques \`a liens arithm\'etiques},    Rend. acc. Lincei ({\bf 4}) 2 56--61. 
  {\bf 76},   5--7.
\bibitem{CMS}  Crstici B., Mitrinovi\'c D. S., S\'andor J.    {\it Handbook of Number Theory I},   Second Edition, (2006),
Springer, Dordrecht, The Netherlands.
\bibitem{CS}  Crstici B., S\'andor J.    {\it Handbook of Number Theory II},    (2004),
Kluwer Academic Publishers, The Netherlands.

\bibitem{D1}  Davenport H. (1937) {\sl On some infinite series involving arithmetical functions}, Q. J. Math. {\bf 8},
8--13. 
\bibitem{D2}  Davenport H. (1937) {\sl On some infinite series involving arithmetical functions II}, Q. J. Math. {\bf 8},
313--320. \bibitem{Ga} G\'al I. S.  (1949)
{\sl A theorem concerning diophantine approximations},  Nieuw Arch. Wiskunde {\bf 23},  13--38.
  \bibitem{Gr} Gronwall T.H.  (1912)
{\sl Some asymptotic expressions in the theory of numbers},  Trans. Amer. Math. Soc. {\bf 8},  118--122.
 \bibitem{HSW}
Haukkanen P., Wang J., Sillanp\"a\"a J. (1997), {\sl On Smith's determinant},    Linear algebra and its Appl. {\bf 258},
  251--269.   \bibitem{H} Hilberdink T.  (2009)
{\sl An arithmetical mapping and applications to $\O$-results for the Riemann zeta function},  Acta Arith. {\bf 139},  341--367.
 \bibitem{HL}   Hong S.   Loewy R. (2004)  {\sl Asymptotic behavior of eigenvalues of greatest common divisor matrices},  Glasgow Math. J.
{\bf 46}, 551--569. 
\bibitem{Ho}  Hooley C.  (1979) {\sl A new technique and its application to the theory of numbers},    Proc. London Math. Soc. (3) {\bf
38}, 115--151. \bibitem{HJ} Horn R. A., Johnson C. R. (1985), {\sl Matrix analysis},   Cambridge University Press, New-York.   
   \bibitem{J} Jacobsthal E. (1957), {\sl \"Uber die gr\"osste ganze Zahl,  II}, Norske Videnkabers Selskab Forhandlinger, Trondheim,   
{\bf 30},
  6--13.
 \bibitem{J}   Jaffard S. (2004)
{\rm  On Davenport expansions}, \emph{Proc. of Symp. in Pure Math.}  {\bf 72.1}, 273--303.
 \bibitem{JW}   Jerosch F., Weyl H. (1909)
{\rm  \"Uber die Konvergenz von Reihen die nach periodischen Funktionen fortschreiten}, \emph{Math. Ann.}  {\bf 66},
p.67.
\bibitem{Jo}  Jordan C. (1881)
{\rm  Sur la s\'erie de Fourier},
\emph{C.R. Acad. Sci. Paris} {\bf 92},  228--230.\bibitem{Li}  Li Z.  (1990)  {\sl  The determinant of GCD matrices},  Linear algebra and its Appl. {\bf 134},
  137--143.
\bibitem{LS}  Lindqvist P., Seip K.  (1998)
{\sl  Note on some greatest common divisor matrices},  {Acta Arith.}  {\bf LXXXIV 2}, 149--154.
\bibitem{MC} McCarthy P. J. {\sl Introduction to Arithmetical Functions}, (Universitext), (1986), Springer-Verlag New-York Inc.
\bibitem{M}  Mansion P. (1878)
{\sl  D\'emonstration d'un th\'eor\`eme relatif \`a un d\'eterminant remarquable},  {Bull. Acad. R. Sc. de Belgique} (2) {\bf 46}, 149--154.
        \bibitem{O}  Olevskii A. M. (1975)    {\sl Fourier series with respect to general orthogonal systems},  Ergebnisse der Mathematik
und ihrer
 Grenzgebiete, Band {\bf 86}.
 \bibitem{R}   Riemann  B.  (1892)
{\sl Gesammelte mathematische Werke und wissenschaftlicher Nachlass}, B. G. Teubner, Leipzig.
\bibitem{S}  Smith H. J. S. (1875/76)  {\sl On the value of certain arithmetical determinant},   Proc. London Math. Soc. {\bf 7},
  208--212.
 \bibitem{Te} {G. Tenenbaum} (2008) {\sl Introduction \`a la th\'eorie analytique et probabiliste des nombres}, Coll. \'Echelles Ed. Belin  
 Paris.
\bibitem{T1}  Tenenbaum G. (1985) {\sl Sur la concentration moyenne des diviseurs}, Comment.  Math. Helv. {\bf 60}, 411--428.
\bibitem{Ti} {Titchmarsh E. C.} [1986]      {\sl The theory of
the Riemann-Zeta function}, Second Edition, revised by Heath-Brown D. R., Oxford
Science Publications.   \bibitem{Va} Varga R.S.: (1965)  {\sl Minimal Ger\v sgorin sets},
  Pacific J. Math.,  {\bf 15} No2, 719--729. 
\bibitem{We} {Weber M.} (2011) {\sl On systems of dilated functions},    C. R. Acad. Sci.
Paris, Sec. 1 {\bf 349}, 1261--1263.
  \bibitem{Wb}  {Weber M.} (2009)   {\sl Dynamical Systems and Processes}, European Mathematical
Society Publishing House, IRMA Lectures
 in Mathematics 
and Theoretical Physics {\bf 14}, xiii+759p. 
\bibitem{Wi}  Wintner, A.  (1944) {\sl Diophantine approximations and Hilbert's space},  Amer.\ J. Math. {\bf 66}, 564--578.
  \bibitem{Z} {Zygmund. A} (2002):  {\sl Trigonometric series}, Third Ed. Vol. {\bf  1}\&{\bf  2} combined,
Cambridge Math. Library, Cambridge Univ. Press. 

     \end{thebibliography}
\end{document}